\documentclass[journal]{IEEEtran}
\usepackage{amssymb, amsthm, bm}
\usepackage[tbtags]{amsmath}
\usepackage{datetime}

\usepackage{graphicx}
\graphicspath{{figures-pdf//}}
\usepackage[caption=false, font=footnotesize, labelfont=rm, textfont=rm]{subfig}

\usepackage{cite} 
\usepackage{url}  
\usepackage[aboveskip=1pt]{subcaption}
\usepackage{diagbox}
\usepackage{multirow, bigstrut}
\usepackage{arydshln} 
\usepackage[mathlines]{lineno}

\usepackage{makecell}
\usepackage{verbatim} 
\usepackage{comment} 

\usepackage{hyperref}
\hypersetup{linktocpage=true, pdfborderstyle={/S/S/W 1}, bookmarksopen=true, 
colorlinks, linkcolor=blue, anchorcolor=blue, citecolor=blue}

\newlength\OneImW
\newlength\BigOneImW
\newlength\twofigwidth
\newtheorem{Corollary}{Corollary}
\newtheorem{theorem}{Theorem}

\newtheorem{lemma}{Lemma}
\newtheorem{Definition}{Definition}

\DeclareMathOperator*{\lcm}{lcm}

\newcommand{\vvert}[0]{{\, \vert \, }}

\newcommand{\Zp}[1]{\mathbb{Z}_{p^{#1}}}

\newcommand{\ord}[0]{\mathrm{ord}}
\newcommand{\bnu}{\boldsymbol{\nu}}

\newcommand{\bv}[0]{\mathbf{v}}
\newcommand{\bx}[0]{\mathbf{x}}
\newcommand{\bM}[0]{\mathbf{M}}
\newcommand{\by}[0]{\mathbf{y}}
\newcommand{\bQ}[0]{\mathbf{Q}}
\newcommand{\bA}[0]{\mathbf{A}}
\newcommand{\bJ}[0]{\mathbf{J}}
\newcommand{\bD}[0]{\mathbf{D}}
\newcommand{\bo}[0]{\mathbf{0}}

\begin{document}

\title{The Graph Structure of a Class of Permutation Maps over Ring $\mathbb{Z}_{p^k}$}

\author{Kai Tan, Chengqing Li
    \thanks{This work was supported by the National Natural Science Foundation of China (no.~92267102).
    }
    \thanks{K. Tan and C. Li are with the School of Computer Science, Xiangtan University, 
Xiangtan 411105, Hunan, China (chengqingg@gmail.com).}
}
    \markboth{IEEE Transactions}{Li \MakeLowercase{\textit{et al.}}}
    \IEEEpubid{\begin{minipage}{\textwidth}\ \\[12pt]\centering
    \\[2\baselineskip]
    0018-9448 \copyright 2024 IEEE. Personal use is permitted, but republication/redistribution requires IEEE permission.\\
    See http://www.ieee.org/publications\_standards/publications/rights/index.html for more information.\\ \today{}\space(\currenttime)
    \end{minipage}
}

\maketitle

\begin{abstract}
Understanding the periodic and structural properties of permutation maps over residue rings such as $\mathbb{Z}_{p^k}$ is a foundational challenge in algebraic dynamics and pseudorandom sequence analysis. Despite notable progress in characterizing global periods, a critical bottleneck remains: the lack of explicit tools to analyze local cycle structures and their evolution with increasing arithmetic precision.
In this work, we propose a unified analytical framework to systematically derive the distribution of cycle lengths for a class of permutation maps over $\mathbb{Z}_{p^k}$. The approach combines techniques from generating functions, minimal polynomials, and lifting theory to track how the cycle structure adapts as the modulus $p^k$ changes. To validate the generality and effectiveness of our method, we apply it to the well-known Cat map as a canonical example, revealing the exact patterns governing its cycle formation and transition.
This analysis not only provides rigorous explanations for experimentally observed regularities in fixed-point implementations of such maps but also lays a theoretical foundation for evaluating the randomness and dynamical behavior of pseudorandom number sequences generated by other nonlinear maps. The results have broad implications for secure system design, computational number theory, and symbolic dynamics.
\end{abstract}


\section{Introduction}

\IEEEPARstart{L}{inear} recurring sequences are a class of sequences where each term is a linear combination of the preceding terms, which are defined by a recurrence relation and initial conditions. The typical sequence generator of linear recurring sequence over the ring $\mathbb{Z}_{p^k}$ include the linear feedback shift register (FSR) \cite{Golic:LFSR:TIT2006, Li:LFSR:TIT2014, Chang:LFSR:TIT2020, Yu:PredictingLFSR:2024}, Chebyshev polynomial \cite{Liao:Che:TC2010, Yoshioka:Chebyshev2k:TCAS2:2016}, discrete Arnold's Cat map\cite{Chen:cat:TIT2012, chenf:cat2:TIT13, chenf:cat:TCS14, Hua:Designing:TCy2018, cqli:Cat:TC22}, R\'enyi map \cite{Addabbo:Reyi:TCSI2007}, inversive congruential pseudorandom number generator \cite{Sole:IPRG:EJC2009, xujun:congru:crypto19}, etc. Linear FSRs can be used to construct maximum-length nonlinear FSRs, essential building blocks for stream ciphers 
\cite{Li:LFSR:TIT2014, Chang:LFSR:TIT2020}. Chebyshev polynomials play a crucial role in the design of public-key algorithms, owing to their distinctive properties and the semi-group characteristics they exhibit \cite{Liao:Che:TC2010, Dariush:Chebyshevprivacy:TII2020}. The Cat map is employed both as a pseudorandom number generator \cite{Hua:Designing:TCy2018, Souza:cat3m:CSII2021, Souza:cat2m:TIM2022} and as a scrambling component in encryption algorithms \cite{Tong:image:ND2015}. In conclusion, these sequences find extensive application across various domains, serving as fundamental components \cite{knuth1985deciphering}.

The period is a significant characteristic of linear recurring sequences. A more extended period implies a more complex and secure sequence, as repeating sequences take longer to emerge. Numerous studies analyze the length of the period.
The period distribution of sequences generated by Chebyshev polynomials over the finite field $\mathbb{Z}_N$ has been analyzed, and public-key algorithms based on these sequences have been optimized \cite{Liao:Che:TC2010}. The distribution of the least period of the generalized discrete Cat map over the Galois ring $\mathbb{Z}_{p^k}$ with different parameters has been demonstrated when $p\geq3$ \cite{Chen:cat:TIT2012}. Subsequently, these analyses further extended to $\mathbb{Z}_{2^{k}}$ \cite{chenf:cat2:TIT13} and $\mathbb{Z}_N$ \cite{chenf:cat:TCS14}, respectively.
In 2016, Yoshioka analyzed the relation of the period of the Chebyshev integer sequence over the ring $\mathbb{Z}_{2^k}$ and revealed the insecure property of the public-key cryptosystem based on the sequence \cite{Yoshioka:Chebyshev2k:TCAS2:2016}.
We reveal the period distribution of the sequence generated by Arnold's Cat map with given parameters over
$\mathbb{Z}_{2^k}$ \cite{cqli:Cat:TC22}. Then the period distribution of Chebyshev integer sequence over the ring $\mathbb{Z}_{p^k}$ has been demonstrated when $p\geq3$ \cite{cqli:Cheby:TIT25}.
These studies use various mathematical approaches. In \cite{Liao:Che:TC2010, Chen:cat:TIT2012, chenf:cat2:TIT13, chenf:cat:TCS14}, these studies combine the generation function and Hensel lifting method, which are based on the minimal polynomial. However, these works primarily analyze the relation between the least period of the Cat map and the Chebyshev polynomial parameters rather than exploring the period distribution across all sequences generated with different initial states. In studying the properties of the period distribution of linear recurring sequences as $k$ increases, many researchers leverage insights from various approaches, such as matrix theory \cite{cqli:Cat:TC22} and analytic number theory \cite{Yoshioka:Chebyshev2k:TCAS2:2016, cqli:Cheby:TIT25,cqli:IPNG:ND25}. Furthermore, the period distribution of different linear recurring sequences should exhibit a consistent property influenced by their ordinary minimal polynomials. This finding is a key contribution of our work.

Arnold's Cat map is a well-known permutation that exhibits a linear recurrence relation over $\mathbb{Z}_N$. It is commonly applied in various research fields, such as quantum chaotic field theories 
quantum Cat map \cite{Kurlberg:Cat:AM05, Axenides:cat:PRE23}, image encryption \cite{liaox:catEncrypt:TITS23, Huang:Arnold:TII23} and the generation of pseudorandom sequences, due to its chaotic behavior in the real number field.
Several methods for generating higher-dimensional Cat maps have been proposed to enhance randomness and expand the parameter space \cite{cqli:Baker:TC25}.

An $n$-dimensional Cat map generation method using Laplace expansions was proposed to iteratively construct an $n$-dimensional Cat matrix \cite{Hua:Designing:TCy2018}. In \cite{Hua:Designing:TCy2018}, the relation between the Lyapunov exponent of the Cat map and its associated Cat matrix is revealed, and a model for constructing $n$-dimensional Cat maps is proposed based on this property. Zhang et al. introduced a method based on Pascal-matrix theory to build an $n$-dimensional Cat matrix \cite{Zhang:2022:Pascal}.
Therefore, Arnold's Cat map is an excellent example of the period distribution of linear recurrence sequence over $\mathbb{Z}_{p^k}$.
The periodic behavior of Arnold's Cat map on a discrete torus is explored in \cite{dyson1992periodCat}.
The period distribution of Arnold's Cat map has been analyzed over different rings \cite{Chen:cat:TIT2012, chenf:cat2:TIT13, chenf:cat:TCS14}. Souza et al. propose a pseudorandom number generator based on the discrete Arnold’s Cat map defined over the integer ring $\mathbb{Z}_{2^m}$ \cite{Souza:cat2m:TIM2022} and $\mathbb{Z}_{3^m}$ \cite{Souza:cat3m:CSII2021}, respectively.
In 2022, we revealed the period distribution of the sequence generated by the 2D generalized discrete Arnold’s Cat map (hereafter referred to as the Cat map for brevity)
with given parameters over $\mathbb{Z}_{2^k}$ \cite{cqli:Cat:TC22}.

In this work, we present properties related to the explicit periodicity analysis of linear recurrence sequences over $\mathbb{Z}_{p^k}$ and disclose the graph structure of these sequences. We analyze the period distribution of sequences generated by Arnold's Cat map over $\mathbb{Z}_{p^k}$ as an example to verify the results.

The remainder of this paper is organized as follows
Section~\ref{preli} reviews the necessary background on linear recurrences and generating functions. Section~\ref{sec:main} establishes the period distribution of a class of permutation maps over $\mathbb{Z}_{p^k}$ and illustrates the results using the Cat Map as a representative example. 
The last section concludes the paper.

\section{Preliminary}
\label{preli}

\begin{table*}[!htb]
\caption{List of some mathematical notations.}
\centering
\setlength{\tabcolsep}{3pt}
    \begin{tabular}{ll|ll}
    \hline
    $\mathbb{Z}$ & The set of all integers.  &$\mathbb{N}^+$  & The set of all positive integers.  \\ \hline
    $\mathcal{Z}_{p^k}$ & The set of non-negative integers less than $p^k$.    & $P'_k(f)$  & The least positive integer $l$ such that $f(t)\mid t^l-1$ in $\mathbb{Z}_{p^k}[t]$.  \\ \hline
    $\Zp{k}$            & The ring of residue classes of the integers modulo $p^k$.
    & $P_k(f)$    & The least positive integer $l$ such that $f(t)\mid t^l-1$ in $\mathbb{R}_{k}[t]$.\\ \hline 
    $\ord_k(x)$     & The least positive number $j$ such that $x^j\equiv 1\pmod p$ in $\mathbb{R}_k$.  & $T_k$    & The least period of $\mathbf{A}(\bx)$ over $\mathbb{Z}^{mn}_{p^k}$.\\ \hline
   $\mathbf{A}^i(\bx)$ & The composition of $\mathbf{A}(\bx)$ with itself $i$ times.
    &$a\mid b$    & The integer $a$ divides the integer $b$. \\ \hline
    $a\equiv b\pmod{q}$ & The base $a$ is congruent to the residue $b$ modulo $q$. &$\gcd(a, b)$ & The greatest common divisor of the integers $a$ and $b$.\\ \hline
   $G_{p^k}$ & The functional graph of the mapping $\bA$.& $\bnu(x)$& The exponent of the highest power of $p$ that divides integer $x$. \\ \hline
   \end{tabular}
\label{tab:list:symbol}
\end{table*}


Given an $n$-dimensional map $\mathbf{P}(\mathbf{a})$ over $\Zp{k}^n$, 
the ring of $n$-dimensional column vectors over $\Zp{k}$, let $S_{\mathbf{a}}=
\{\mathbf{P}^i(\mathbf{a})\}_{i=0}^\infty$
denote the sequence generated by iterating $\mathbf{P}$ from the initial state $\mathbf{a}$, where $\mathbf{P}^i$ represents 
the $i$-fold application of $\mathbf{P}$, and $\mathbf{a}=(a_1, a_2, \cdots, a_n)^\intercal\in\Zp{k}^n$.
If the sequence $S_{\mathbf{a}}$ satisfies the recurrence relation
\begin{equation}\label{eq:2:cP}
\mathbf{P}^i(\mathbf{a})+c_1 \mathbf{P}^{i-1}(\mathbf{a})+\cdots+c_L \mathbf{P}^{i-L}(\mathbf{a})=0
\end{equation}
for all $i\ge L$, it is called an $L$-order linear recurrence over $\Zp{k}^n$.
The characteristic polynomial of the sequence is given by
$f_{\rm c}(t)=1+c_1 t+\cdots+c_{L-1} t^{L-1}+ c_L t^L$.
Note that the operators in \eqref{eq:2:cP} can be compactly represented as $f_c(\mathbf{P})$.

Let the generating function of the vector-valued sequence be defined as $G_{\mathbf{a}}(t)=\sum_{i=0}^{\infty} d(\mathbf{P}^i(\mathbf{a}))t^i$, where
$d(\bx)=\sum_{i=1}^n x_i p^{i-1}$ is a function applied element-wise to each coordinate of the vector $\bx=(x_1, x_2, \cdots, x_n)^\intercal\in\Zp{k}^n$. 
Then, one has 
\begin{multline*}
  G_{\mathbf{a}}(t) \cdot f_{\rm c}(t) = \sum_{i=0}^{L-1} \left( \sum_{j=0}^{i} c_j d(\mathbf{P}^{i-j}(\mathbf{a})) \right) t^i \\
    \quad + \sum_{i=L}^{\infty} \left( \sum_{j=0}^{L} c_j d(\mathbf{P}^{i-j}(\mathbf{a})) \right) t^i. 
\end{multline*}
Since 
$\sum_{j=0}^{L} c_j d(\mathbf{P}^{i-j}(\mathbf{a}))=d(\sum_{j=0}^{L} c_j \mathbf{P}^{i-j}(\mathbf{a}))$ for any $i\ge L$, one has
\begin{equation}\label{eq:PRE:gt}
G_{\mathbf{a}}(t) \cdot f_{\rm c}(t)=g(t)
\end{equation}
from \eqref{eq:2:cP}, where $g(t)= \sum_{i=0}^{L-1} \left( \sum_{j=0}^{i} c_j d(\mathbf{P}^{i-j}(\mathbf{a})) \right) t^i$. Moreover, 
it follows from \eqref{eq:2:cP} that
\begin{equation}\label{eq:PRE:Ht}
H(t)=(1-t^{T_s})\cdot G_{\mathbf{a}}(t)=\sum_{i=0}^{T_s-1}d(\mathbf{P}^i(\mathbf{a}))t^i, 
\end{equation}
where $T_s$ denotes the least period of the sequence $S_{\mathbf{a}}$. 
There is $f(t)=b_0+b_1 t+\cdots+b_{m-1}t^{m-1}+b_mt^m$ such that
\begin{equation}\label{eq:ft}
f(t)=\frac{f_{\rm c}(t)}{\gcd(g(t), f_{\rm c}(t))}.
\end{equation}
Since $f(t)$ is the polynomial of the smallest degree among all characteristic polynomials of the sequence $S_{\mathbf{a}}$, it is referred to as the \textit{minimal polynomial} of the sequence.

The minimal polynomial can also be expressed using the generating function.
From~\eqref{eq:PRE:gt} and~\eqref{eq:PRE:Ht}, it follows that
\[
\frac{f_{\rm c}(t)}{g(t)}=\frac{1-t^{T_s}}{H(t)}, 
\]
which, together with \eqref{eq:ft}, implies that the minimal polynomial is given by
$f(t)=\frac{1-t^{T_s}}{\gcd(H(t), 1-t^{T_s})}$, removing the redundant factors to capture only the essential periodicity.
Thus, as stated in Lemma~\ref{le:2:T=P_e}, the least period of a sequence is determined by the order of the corresponding minimal polynomial. When $k=1$, $\mathbb{Z}_p$ is a finite field, the order of the 
polynomial is quantified as Lemma~\ref{le:2:P_1} shown.

\begin{lemma}\label{le:2:T=P_e}
For a periodic sequence $S_{\mathbf{a}}$ over $\Zp{k}^n$, the order of its minimal polynomial $f(t)$ over $\mathbb{Z}_{p^k}$, denoted as $P'_k(f)$, is equal to the least period of $S_{\mathbf{a}}$.
\end{lemma}
\begin{proof}
See proof in \cite{Linear:Zierler:1959}.
\end{proof}

\begin{lemma}\label{le:2:P_1}
Let \( f(x) \) be a polynomial over \( \mathbb{F}_q \) with a nonzero constant term and assume that  
\[ f(x) = f_1(x)^{e_1} f_2(x)^{e_2} \cdots f_r(x)^{e_r}, \]  
where \( f_1(x), f_2(x), \ldots, f_r(x) \) are \( r \) distinct irreducible polynomials in \( \mathbb{F}_q[x] \) and \( e_1, e_2, \ldots, e_r \) are \( r \) positive integers, then  
\[ 
p(f) = \lcm(p(f_1), p(f_2), \ldots, p(f_r)) \cdot p^t,
\]
where $t=\min\{t\vvert p^t \geq \max(e_1, e_2, \ldots, e_r)\}$.
\end{lemma}
\begin{proof}
See proof in \cite{Wan:Lecture:2003}.
\end{proof}

It follows from \eqref{eq:2:cP} and the definition of $f(t)$ that
\begin{equation}\label{eq:2:An2x}
\mathbf{P}^{i}(\mathbf{a})=-\sum_{j=1}^{m}b_i\mathbf{P}^{i-j}(\mathbf{a}).
\end{equation}
To disclose the relation between $f(t)$ and $\mathbf{P}(\mathbf{a})$, construct a map
\begin{equation}\label{eq:3:ma=}
\mathbf{A}(\bx)=
\bM\cdot \bx
\end{equation}
over $\mathbb{Z}^{mn}_{p^k}$, 
\begin{equation*}
\bM=
\begin{bmatrix}
    \bo      & \mathbf{I}    &\cdots & \bo\\
    \vdots          & \vdots        &\vdots & \vdots    \\
    \bo      & \bo    &\cdots & \mathbf{I}\\
  -b_1\mathbf{I}    & -b_2\mathbf{I}  &\cdots & -b_{m}\mathbf{I}
\end{bmatrix}_{mn\times mn}, 
\end{equation*}
and $\mathbf{I}$ denotes the $n\times n$ identity matrix.
The minimal polynomial of sequence $\{\bM^i\cdot \bx\}_{i\geq 0}$ is also $f(t)$, and its least period over $\mathbb{Z}^{mn}_{p^k}$
is denoted by $T_k$.
For an initial state $\bx_0$, let $T_k(\bx_0)$ be the least period of $\{\mathbf{A}^i(\bx)\}_{i\geq 0}$ over $\Zp{k}$. Note that
$T_k{(\bx_0)}$ is a divisor of $T_k$.

For any state $\mathbf{a}$ in $\mathbb{Z}^n_{p^k}$, one can form the state
\[
\bx =(\mathbf{a}^\intercal, (\mathbf{P}(\mathbf{a}))^\intercal, (\mathbf{P}^2(\mathbf{a}))^\intercal, \cdots, (\mathbf{P}^{m-1}(\mathbf{a}))
^\intercal)^\intercal\in \mathbb{Z}^{mn}_{p^k}.
\]
However, given an arbitrary state $\bx$ in $\mathbb{Z}^{mn}_{p^k}$, there may not exist a corresponding state $\mathbf{a}$ in $\mathbb{Z}^n_{p^k}$ that satisfies the above relationship. Consequently, one may first analyze the structure of the functional graph generated by $\mathbf{A}(\bx)$ and then derive that of $\mathbf{P}(\mathbf{a})$ by truncation.

To enhance the readability of this paper, Table~\ref{tab:list:symbol} enumerates the notations recurrently used throughout the paper.

\section{The period distribution of a permutation map over $\Zp{k}$}
\label{sec:main}

\subsection{The Period Dstribution of $\mathbf{A}(\bx)$ over $\Zp{k}^{mn}$}
\label{subsec:3A}

Given the minimal polynomial $f(t)$, the companion matrix $\bM$ in (\ref{eq:3:ma=}) is uniquely determined.
The analysis of the period distribution of map $\mathbf{A}(\bx)$ relies on the
order of $f(t)$ and its factorization in $\Zp{k}[t]$.
When $f(t)$ is irreducible, one passes to the minimal ring extension $\mathbb{R}_k$ of $\mathbb{Z}_{p^k}$, over which $f(t)$ splits into linear factors. 
Lemma~\ref{le:3:orderZ=R} then shows that the order of $f(t)$ in $\mathbb{Z}_{p^k}[t]$ coincides with its order in $\mathbb{R}_k[t]$. Moreover, Lemma~\ref{le:3A:P_kfp} establishes that, beyond a threshold $k_{\rm s}$, the order in $\mathbb{R}_k[t]$ increases by a factor of $p^{k-k_{\rm s}}$ relative to $P_1(f)$.

\begin{lemma}\label{le:3:orderZ=R}
The order of $f(t)$ in $\mathbb{Z}_{p^k}[t]$, denoted by $P'_k(f)$, is equal to its order in $\mathbb{R}_{k}[t]$, denoted by $P_k(f)$.
\end{lemma}
\begin{proof}
First, by definition of $P'_k(f)$, there exists a polynomial $g(t)\in\mathbb{Z}_{p^k}[t]$ such that 
\begin{equation}\label{eq:3A:gftP}
    g(t)f(t)= t^{P'_k(f)}-1
\end{equation}
in $\mathbb{Z}_{p^k}[t]$. Since $\mathbb{R}_k$
is an extension ring of $\mathbb{Z}_{p^k}$, there is a map $\varphi:\mathbb{Z}_{p^k}[t]\rightarrow\mathbb{R}_k[t]$ that maps each polynomial to itself with coefficients reinterpreted in $\mathbb{R}_k[t]$. Since $\varphi$ is a ring homomorphism, it preserves polynomial identities. Thus, Eq.~\eqref{eq:3A:gftP} holds in $\mathbb{R}_k[t]$. By definition of $P_k(f)$, one has
\begin{equation}\label{eq:3A:P'=P}
    P_k(f)\mid P'_k(f).
\end{equation}
Next, by definition of $P_k(f)$, there is $g'(t)\in\mathbb{R}_{k}[t]$ such that \begin{equation}\label{eq:3A:g'f}
    g'(t)f(t) = t^{P_k(f)} - 1.
\end{equation}
in $\mathbb{R}_{k}[t]$. By the construction of $\mathbb{R}_k$, there exists a minimal generating set $\{1, \beta_1, \beta_2, \cdots, \beta_s\}$, where the set $\{\beta_i\}_{i=1}^s$ is algebraically independent over $\mathbb{Z}_{p^k}$. 
Then, write
\begin{equation*}
    g'(t) = g'_0(t)+\sum_{i=1}^s \beta_i g'_i(t),
\end{equation*}
where $g'_i(t)\in\Zp{k}[t]$. 
Substituting this expression into \eqref{eq:3A:g'f} yields
\begin{equation*}
    g'_0(t)f(t) + \sum_{i=1}^s \beta_i g'_i(t)f(t) = t^{P_k(f)} - 1.
\end{equation*}
Let the left-hand side of the above equation 
can be expressed as 
$\sum_{i=0}b_it^i$,
where $b_i=c_{i,0} +\sum_{j=1}^s c_{i,j}\beta_j$ and $c_{i,0}\in \mathbb{Z}_{p^k}$.
Then $b_0=c_{0,0}=1, b_{P_k(f)}=c_{P_k(f),0}=1$ and $b_i=c_{i,0}=0$ for any $i\not\in\{0, P_k(f)\}$.
Assume $\beta_{j_1}g'_{j_1}(t)f(t)\neq 0$. There is $c_{i_1,j_1}\neq 0$.
Since $\beta_i$ cannot be expressed as a linear combination of other generators, 
there is $b_{i_1}-c_{i_1,0}\neq 0$, which leads to a contradiction.
So $g'_0(t)f(t)=t^{P_k(f)}-1$ in $\mathbb{Z}_{p^k}[t]$, namely $P'_k(f)\mid P_k(f)$. Combining this with relation~\eqref{eq:3A:P'=P}, this lemma holds.
\end{proof}

\begin{lemma}\label{le:3A:P_kfp}
If the monic polynomial $f(t)$ has no root of value zero in $\mathbb{R}_1$, then one has
  \begin{equation}\label{eq:f:pef-nomul}
      P_k(f)=
      \begin{cases}
            P_1(f)                & \mbox{if } k<k_{\rm s}; \\
            p^{k-k_{\rm s}}P_1(f) & \mbox{if } k\ge k_{\rm s}, 
      \end{cases}
  \end{equation}
where
\begin{equation}\label{eq:3A:f_hatkf_define}
    k_{\rm s}=\max\{k\mid P_k(f)=P_1(f)\}.
\end{equation}
\end{lemma}
\begin{proof}
If $k\leq k_{\rm s}$, \eqref{eq:f:pef-nomul} holds by~\eqref{eq:3A:f_hatkf_define}. 
If $k\geq k_{\rm s}$, one can prove 
\begin{equation}\label{eq:3A:ftmnom}
\left\{
\begin{split}
    P_{k_{\rm s}+i}(f)   & = p^iP_1(f);\\
    P_{k_{\rm s}+i+1}(f) &\neq p^iP_1(f)    
\end{split}
\right.
\end{equation}
by mathematical induction on $i$. 
When $i=0$, it follows from the definition of $k_{\rm s}$ that \eqref{eq:3A:ftmnom} holds. Assume that \eqref{eq:3A:ftmnom} holds for $i\leq s$. When $i=s+1$, in the ring $\mathbb{R}_{{k_{\rm s}+s+1}}[t]$, there is $v(t)\not\in(p)$ such that $f(t)\nmid v(t)$ but
\begin{equation*}
    f(t) \mid (t^{p^{s}P_{1}(f)}-1+p^{k_{\rm s}+s} \cdot v(t)).
\end{equation*}
Note that $(t^{p^{s}P_{1}(f)}-1)(\sum_{i=0}^{p-1} t^{i \cdot p^{s}P_{1}(f)}) = t^{ p^{s+1}P_{1}(f)}-1$, one can construct a polynomial
\begin{equation}\label{eq:3A:qt}
    q(t) = p^{k_{\rm s}+s} \cdot v(t) r(t)
\end{equation}
where $r(t)=\sum_{i=0}^{p-2}(p-1-i)t^{i\cdot p^{s}P_{1}(f)}$.
It follows that 
\begin{multline}\label{eq:3A:t-1rt}
\left(\sum_{i=0}^{p-1} t^{i \cdot p^{s}P_{1}(f)} - q(t)\right)  \left(t^{p^{s}P_{1}(f)} - 1 + p^{k_{\rm s}+s} \cdot v(t)\right) \\ 
= t^{p^{s+1}P_{1}(f)} - 1+ p^{k_{\rm s}+s+1} \cdot v(t).
\end{multline}
Hence $f(t)\mid t^{p^{s+1}P_{1}(f)} - 1$ in $\mathbb{R}_{{k_{\rm s}+s+1}}[t]$, namely 
$P_{k_{\rm s}+s+1}=p^{s+1}P_{1}(f)$.
In $\mathbb{R}_{{k_{\rm s}+s+2}}[t]$, it follows that $v'(t)\not\in(p)$ such that $f(t)\nmid v(t)$ and
\begin{equation*}
    f(t) \mid t^{p^{s}P_{1}(f)}-1+p^{k_{\rm s}+s} \cdot v'(t).
\end{equation*}
Similarly to~\eqref{eq:3A:t-1rt}, there is $f(t)\mid t^{p^{s+1}P_{1}(f)} - 1+ p^{k_{\rm s}+s+1} \cdot v'(t)$. Since $f(t)\nmid v(t)$, one obtain $f(t)\nmid t^{p^{s+1}P_{1}(f)}-1$.
So \eqref{eq:3A:ftmnom} holds for $i=s+1$.
This completes the proof of \eqref{eq:3A:ftmnom}, which means that 
\begin{equation*}
\left\{
    \begin{split}
        P_{k}(f) & \nmid p^{k-k_{\rm s}-1} P_{1}(f); \\
        P_{k}(f) & \mid  p^{k-k_{\rm s}}P_{1}(f).
    \end{split}
    \right.
\end{equation*}
Thus, this lemma holds.
\end{proof}

A monic polynomial $f(t)$ can be represented as
\[
f(t)=(t-\alpha_{1, k})(t-\alpha_{2, k})\cdots(t-\alpha_{m, k})
\]
in $\mathbb{R}_{{k}}[t]$.
Then these roots should reduce modulo $p$ to the roots $\{\alpha_i\}_{i=1}^m$ in $\mathbb{R}_{{1}}[t]$, namely 
\begin{equation}\label{eq:3A:alpha-}
    \alpha_{i, k}-\alpha_i\in(p). 
\end{equation}
Based on the analysis of root lifting between the rings $\mathbb{R}_k$, the specific value $k_{\rm s}$ can be explicitly determined. When the polynomial $f(t)$ has no non-zero multiple roots in $\mathbb{R}_1$, Lemma~\ref{le:3:hensem3} provides the exact formulation for $k_{\rm s}$.

\begin{lemma}\label{le:3:hensem3}  
If the monic polynomial $f(t)$ has exactly $m$ distinct non-zero roots 
in $\mathbb{R}_1$, (\ref{eq:f:pef-nomul}) holds
with $k_{\rm s}=\min\limits_{1\leq i\leq m}f_{\rm k}(\alpha_i)$, 
where $\alpha_i$ is the $i$-th root, 
\begin{equation}\label{eq:3:ealphaj}
f_{\rm k}(\alpha_i)=\max\left\{ k \vvert \ord_{k}\left(g^{(k-1)}(\alpha_i)\right)=\ord_1(\alpha_i) \right\}, 
\end{equation}
and
$g^{(n)}(x)$ indicates iterating 
\begin{equation}\label{eq:2:alphai+1}
g(x)=x-f'(x)^{-1}\cdot f(x)
\end{equation}
$n$ times.
\end{lemma}
\begin{proof}
Since $f(t)$ has no non-zero multiple root in $\mathbb{R}_1$ and~\eqref{eq:3A:alpha-}, the same holds in $\mathbb{R}_k$ for any $k$. 
Then, $f(t)=\prod\limits_{i=1}^m f_i$, where 
each linear factor $f_i=t-\alpha_{i,k}$ is pairwise coprime.
Let 
\begin{equation}\label{eq:3A:l}
l=\lcm(P_k(f_1), P_k(f_2), \cdots, P_k(f_m)), 
\end{equation}
then $f_i\mid t^{P_k(f_i)}-1 \mid t^l-1$. From $f_i$ being pairwise coprime, one has $f\mid t^l-1$ by the Chinese Remainder Theorem, namely $P_k(f)\mid l$. Otherwise, since $P_k(f_i) \mid P_k(f)$ for any $i$, one has $l\mid P_k(f)$. It follows from \eqref{eq:3A:l} that 
\begin{equation}\label{eq:3:P_efnomul}
P_k(f)=\lcm(P_k(f_1), P_k(f_2), \cdots, P_k(f_m)).
\end{equation}

In $\mathbb{R}_k$, since $\mathbb{R}_k/(p)\cong\mathbb{R}_1$, one has $\alpha_i\in\mathbb{R}_k$ and $f(\alpha_i)\in(p)$. As $f(t)$ has no multiple root in $\mathbb{R}_1$, one has 
\begin{equation}\label{eq:3A:f'}
    f'(\alpha_i)\not\in(p). 
\end{equation}
Let $\alpha^{(n)}_i=g^{(n)}(\alpha_i)$. It follows from \eqref{eq:2:alphai+1} that $\alpha^{(n)}_i-\alpha_i\in(p)$.  
Then $f'(\alpha^{(n)}_i)-f'(\alpha_i)\in(p)$ by Taylor's formula. Referring to \eqref{eq:3A:f'}, one has $f'(\alpha^{(n)}_i)\not\in(p)$.
Then, from \eqref{eq:2:alphai+1}, one has 
\begin{multline*}
f(\alpha^{(n+1)}_i)=f(\alpha^{(n)}_i)+f'(\alpha^{(n)}_i)(-f'(\alpha^{(n)}_i)^{-1}f(\alpha^{(n)}_i)+ \\ \sum^m_{j=2}f^{(j)}(\alpha^{(n)}_i)(-f'(\alpha^{(n)}_i)^{-1}f(\alpha^{(n)}_i))^j
\end{multline*}
by Taylor's formula.
It follows that $f_i=t-\alpha^{(n+1)}_i$ is a factor of $f(t)$ in $\mathbb{R}_n$, then 
\begin{equation}\label{eq:ord=pe}
\ord_k(g^{(k)}(\alpha_i))=P_k(f_i).
\end{equation}
Combining this equation with \eqref{eq:3:ealphaj}, one has 
\begin{equation*}
\left\{
    \begin{split}
        P_{f_{\rm k}(\alpha_i)}(f_i)   &=P_1(f_i);\\
        P_{f_{\rm k}(\alpha_i)+1}(f_i) &\neq P_1(f_i).
    \end{split}
    \right.
\end{equation*}
Referring to \eqref{eq:3A:l}, there is $f_{\rm k}(\alpha_{i_1})=\min\limits_{1\leq i\leq m} f_{\rm k}(\alpha_{i_1})$ such that 
\begin{equation*}
\left\{
    \begin{split}
        P_{f_{\rm k}(\alpha_{i_1})}(f)   &=P_1(f);\\
        P_{f_{\rm k}(\alpha_{i_1})+1}(f) &\neq P_1(f).
    \end{split}
    \right.
\end{equation*}
According to \eqref{eq:3A:f_hatkf_define}, this lemma holds.
\end{proof}

For $f(t)$ with a multiple root $\alpha$ over $\mathbb{R}_1$, one has $\hat{f}_k\equiv g^a \pmod{p}$ in $\mathbb{R}_k[t]$, 
where $\hat{f}_k=(t-\alpha_{1, k})(t-\alpha_{2, k})\cdots(t-\alpha_{a, k})$.
Lemma~\ref{le:3A:R1pf'} determines the least period of $f(t)$ in $\mathbb{R}_{1}$.
Lemma~\ref{le:3A:gmpkfm<p} describes the regulation of the change of least period of $f(t)$ in $\mathbb{R}_{k}$ as $k$ increases when $m\leq p$. Lemma~\ref{le:3A:gmpkfm>p} addresses the scenario for $m>p$. The combined insights from Lemmas~\ref{le:3:hensem3},~\ref{le:3A:gmpkfm<p} and~\ref{le:3A:gmpkfm>p} are further distilled in Theorem~\ref{coro:2:e>1b=1}, which governs the regulation of the change of least period of $f(t)$ in $\mathbb{R}_{k}$ as $k$ increases.

\begin{lemma}\label{le:3A:R1pf'}
If the monic polynomial $f(t)$ has a nonzero root $\alpha$ of multiplicity $a$ in $\mathbb{R}_1$, then one has
    \begin{equation}\label{eq:peh:g^s}
       P_1((t-\alpha)^a)=p^c \cdot P_1(t-\alpha), 
    \end{equation}
where $c=\min\{k\ \vert \ a \leq p^k\}$.
\end{lemma}
\begin{proof}
In $\mathbb{R}_1[t]$, $(t-\alpha) \mid t^{P_1(t-\alpha)}-1$. 
Since $t^{P_1(t-\alpha)}-1$ is coprime with its derivative, it has only simple roots. Consequently, $(t-\alpha)^a\mid (t^{P_1(t-\alpha)}-1)^{p^c}$ but $(t-\alpha)^a\nmid (t^{P_1(t-\alpha)}-1)^{p^{c-1}}$. It follows that $P_1((t-\alpha)^a)\mid p^c\cdot P_1(t-\alpha)$ but $P_1((t-\alpha)^a)\nmid p^{c-1}\cdot P_1(t-\alpha)$. 
Thus, this lemma holds.
\end{proof}

\begin{lemma}\label{le:3A:gmpkfm<p}
If the monic polynomial $f(t)$ has a nonzero root $\alpha$ of multiplicity $a\leq p$ in $\mathbb{R}_k$, then one has
    \begin{equation}\label{eq:peh:g^s}
       P_k((t-\alpha)^a)=p^{k}\cdot P_1(t-\alpha).
    \end{equation}
\end{lemma}
\begin{proof}
In $\mathbb{R}_1$, it follows from $a\leq p$ and Lemma~\ref{le:3A:R1pf'} that
$P_1(\hat{f})=p\cdot P_1(t-\alpha)$, where $\hat{f}=(t-\alpha)^a$. Then, referring to Lemma~\ref{le:3A:P_kfp}, one has
\begin{equation}\label{eq:3:P_efmidpe}
    P_k(\hat{f}) \mid p^{k}\cdot P_1(g).
\end{equation}
Next, since $\hat{f}$ is a factor of $f(t)$, there exist distinct indices
$i_1, i_2\in \{1, 2, \cdots, m\}$ such that $\alpha_{i_1, 2}-\alpha\in(p)$ and $\alpha_{i_1, 2}-\alpha\in(p)$.
Consequently, one has 
\begin{equation}\label{eq:3A:alpha_{i_1k}}
    \alpha_{i_1, 2}\not\in(p).
\end{equation}
and
\begin{equation}\label{eq:3A:alphai12}
(t-\alpha_{i_1, 2})(t-\alpha_{i_2, 2})\mid t^{P_2(\hat{f})}-1.
\end{equation}
Since $\alpha_{i_1, 2}^{P_2(\hat{f})}=1$,
\begin{equation}\label{eq:3A:t=alphai1}
t^{P_2(\hat{f})}-1=(t-\alpha_{i_1, 2})\sum^{P_2(\hat{f})}_{i=1}\alpha_{i_1, 2}^{i-1}t^{P_2(\hat{f})-i}.
\end{equation}
Then, it follows from the Division algorithm that 
    \begin{multline}\label{eq:3A:remainer}
        \sum^{P_2(\hat{f})}_{i=1}\alpha_{i_1, 2}^{i-1}t^{P_2(\hat{f})-i}=\\
        (t-\alpha_{i_2, 2})\left(\sum^{P_2(\hat{f})-2}_{i=0}d_i\cdot t^{P_2(\hat{f})-2-i}\right)+d_{P_2(\hat{f})-1}, 
    \end{multline}
where
\begin{equation}\label{eq:3A:ds}
d_s=\alpha_{i_2, 2}^s+\alpha_{i_2, 2}^{s-1}\alpha_{i_1, 2}+\cdots+\alpha_{i_1, 2}^s.
\end{equation}
Since $d_s$ is the constant term, one has $d_{P_2(\hat{f})-1}\nmid t-\alpha_{i_2, 2}$.
Combining \eqref{eq:3A:alphai12}, \eqref{eq:3A:t=alphai1}, and~\eqref{eq:3A:remainer}, one can obtain 
\begin{equation}\label{eq:3:dpef-1}
d_{P_2(\hat{f})-1}\in (p^2).
\end{equation}
Thus, from
\begin{equation*}
d_{P_2(\hat{f})-1}(\alpha_{i_2, 2}-\alpha_{i_1, 2})=\alpha_{i_2, 2}^{P_2(\hat{f})}-\alpha_{i_1, 2}^{P_2(\hat{f})}, 
\end{equation*}
 the analysis of $f_{i_1}$, $f_{i_2}$ can be divided into two cases by the relation between $\alpha_{i_2, 2}$ and $\alpha_{i_1, 2}$:
\begin{itemize}
    \item $\alpha_{i_2, 2}=\alpha_{i_1, 2}$: it follows from \eqref{eq:3A:ds} that $d_{P_2(\hat{f})-1}=P_2(\hat{f})\alpha_{i_2, 2}^{P_2(\hat{f})-1}$. Since $t\nmid g$, one has $\alpha\not\in(p)$, then $\alpha_{i_2, 2}^{P_2(\hat{f})-1}\not\in(p)$. Referring to \eqref{eq:3:dpef-1}, one can obtain $P_k(f_{i_1}f_{i_2})\in (p^2)$.
    
    \item $\alpha_{i_2, 2}\neq \alpha_{i_1, 2}$: one can derive
    \begin{equation*}
    \begin{split}
        d_{P_2(\hat{f})-1}&=\frac{\alpha_{i_2, 2}^{P_2(\hat{f})}-\alpha_{i_1, 2}^{P_2(\hat{f})}}{\alpha_{i_2, 2}-\alpha_{i_1, 2}}\\
                    &=\frac{(\alpha_{i_2, 2}-\alpha_{i_1, 2}+\alpha_{i_1, 2})^{P_2(\hat{f})}-\alpha_{i_1, 2}^{P_2(\hat{f})}}{\alpha_{i_2, 2}-\alpha_{i_1, 2}}\\
                    &=\sum_{i=1}^{P_2(\hat{f})} \binom{i}{P_2(\hat{f})}\alpha_{i_1, 2}^{P_2(\hat{f})-i}(\alpha_{i_2, 2}-\alpha_{i_1, 2})^{i-1}, 
    \end{split}
    \end{equation*}
    Combining the conditions $\alpha_{i_2, 2}-\alpha_{i_1, 2}\in (p)$ and $\alpha_{i_1, 2}\not\in(p)$ with \eqref{eq:3:dpef-1}, one obtains $P_2(\hat{f})\in (p)$ and
    \begin{equation*}
        \sum_{i=1}^{2} \binom{i}{P_2(\hat{f})}\alpha_{i_1, 2}^{P_2(\hat{f})-i}(\alpha_{i_2, 2}-\alpha_{i_1, 2})^{i-1}\in(p^2).
    \end{equation*}
    This expression can be simplified to $P_k(f)(\alpha_{i_1, 2}-(P_k(f)-1)(\alpha_{i_2, 2}-\alpha_{i_1, 2}))\in(p^2)$. Since $\alpha_{i_1, 2}-\alpha\in(p)$, one has $\alpha_{i_2, 2}-\alpha_{i_1, 2}\in(p)$. Combining this with \eqref{eq:3A:alpha_{i_1k}} and $P_k(f)\in(p)$, one obtains $P_2(\hat{f})\in (p^2)$.
\end{itemize}
Now, one can conclude that $P_2(\hat{f})\nmid p\cdot P_1(g)$ over $\mathbb{R}_2$.
Referring to Lemma~\ref{le:3A:P_kfp} and \eqref{eq:3:P_efmidpe}, this lemma holds.
\end{proof}

\begin{lemma}\label{le:3A:gmpkfm>p}
If the monic polynomial $f(t)$ has a nonzero root $\alpha$ of multiplicity $a>p$ in $\mathbb{R}_k$, then one has
    \begin{equation}\label{eq:peh:g^s}
       P_k((t-\alpha)^a)=p^{k+c-1}\cdot P_1(t-\alpha), 
    \end{equation}
    where $c=\min\{k\ \vert \ a \leq p^k\}$.
\end{lemma}
\begin{proof}
This lemma can be proved by mathematical induction on $c$. When $c=0$, there is $a=1$. Referring to Lemma~\ref{le:3A:gmpkfm<p}, \eqref{eq:peh:g^s} holds. 
Assume that \eqref{eq:peh:g^s} holds for $c\leq s$.
When $c=s+1$, in $\mathbb{R}_1$, referring to Lemma~\ref{le:3A:R1pf'} and $a>p$, one has \begin{equation}\label{eq:3A:ftppk-1} 
    P_1(\hat{f})=p^{s+1}\cdot P_1(g)
\end{equation}
 and 
\begin{equation}\label{eq:3A:p1hatf}
P_1(\hat{f}_1)=p^{s}P_1(g), 
\end{equation}
where $\hat{f}_1=(t-\alpha_{k, 1})(t-\alpha_{k, 2})\cdots (t-\alpha_{k, p^{s}})$. 
In $\mathbb{R}_2$, since \eqref{eq:3A:p1hatf} holds when $c=s$, there is $\hat{v}(t)\not\in(p)$ such that $\hat{f}_1\mid t^{p^{s}\cdot P_1(g)}-1+p\hat{v}(t)$. 

Similar to \eqref{eq:3A:t-1rt}, there is 
\begin{equation}\label{eq:3A:hatr2t}
    \hat{r}(t)=\sum_{i=0}^{p-2}(p-1-i)t^{i\cdot p^{s}P_1(g)}
\end{equation}
such that
\begin{multline}\label{eq:3A:cfp-m}
\left(\sum_{i=0}^{p-1} t^{i \cdot p^{s}P'_{1}(g)} + p\hat{v}(t)\hat{r}(t)\right)\left(t^{p^{s}P_1(g)}-1-p\hat{v}(t)\right)=\\
t^{p^{s+1}P_1(g)} - 1.
\end{multline}
Referring to Lemma~\ref{le:3A:R1pf'}, there is $\alpha_{i, 2}$ for any $i\in\{p^s+1, p^s+2, \cdots, a\}$ such that $(t-\alpha_{i, 2})\hat{f}_1\nmid t^{p^s\cdot P_1(g)}-1$ 
in $\mathbb{R}_1$.
Then, in $\mathbb{R}_2$, one obtains $(t-\alpha_{i, 2})\hat{f}_1\nmid t^{p^s\cdot P_1(g)}-1+\hat{v}(t)$.
Assume $P_2(\hat{f})=p^{s+1}P_1(g)$. In $\mathbb{R}_2$, combining with \eqref{eq:3A:cfp-m}, one can obtain
\[
(t-\alpha_{i, 2})\mid\left(\sum_{i=0}^{p-1} t^{i \cdot p^{s}P'_{1}(g)} -p\hat{v}(t)\hat{r}(t)\right).
\]
This implies
\[
\left(\sum_{i=0}^{p-1} \alpha_{i, 2}^{i \cdot p^{s}P'_{1}(g)} - p\hat{v}(\alpha_{i, 2})\hat{r}(t\alpha_{i, 2})\right)=0
\]
It follows from the definition of $P_1(g)$ and $\alpha_{i, 2}-\alpha\in(p)$ that $\alpha_{i, 2}^{p\cdot P_1(g)}=1$.
Substituting into the right-hand side of the above equation yields a simplified result that equals $p$, which leads to a contradiction.
Hence, $\hat{f}\nmid t^{p^{s+1}P_1(g)}-1$ in $\mathbb{R}_2$. It follows from \eqref{eq:3A:ftppk-1} that $f\mid t^{p^{s+2}P_1(g)}-1$, then $P_2(\hat{f})=p^{s+2}P_1(g)$.
Applying Lemma~\ref{le:3A:P_kfp}, one concludes that \eqref{eq:peh:g^s} holds for $c=s+1$. This completes the proof of \eqref{eq:peh:g^s}.
\end{proof}


\begin{theorem}\label{coro:2:e>1b=1}
If the monic polynomial $f(t)$ has exactly $m$ non-zero roots in $\mathbb{R}_1$, (\ref{eq:f:pef-nomul}) holds with
\begin{equation*}
 k_{\rm s}=
 \begin{cases}
     1    & \mbox{if } \prod_{i=1}^l a_i>1;\\
 \min\limits_{1\leq i\leq m} f_{\rm k}(\alpha_i) & \mbox{if } \prod_{i=1}^l a_i=1,    
 \end{cases}
\end{equation*}
where $a_i$ is the multiplicity of the $i$-th root $\alpha_i$ of $f(t)$, $l$ is the number of different roots of $f(t)$, 
\begin{equation*}
f_{\rm k}(\alpha_i)=\max\left\{ k \vvert\ord_{k}\left(g^{(k-1)}(\alpha_i)\right)=\ord_1(\alpha_i) \right\}, 
\end{equation*} 
and $m\leq p$.
\end{theorem}
\begin{proof}
The monic polynomial $f(t)$ can be factorized as $f(t)=g_1^{a_1}(t)g_2^{a_2}(t)\cdots g_l^{a_l}(t)$ in $\mathbb{R}_1$, where the factors $\{g_i(t)\}_{i=1}^l$ are pairwise coprime.
Therefore, by the multiplicity of the period function, $P_k(f)=\lcm\{ P_k(g_i^{a_i}) \mid 1\leq i\leq l \}$.
Finally, Lemmas \ref{le:3:orderZ=R}, \ref{le:3:hensem3} and \ref{le:3A:gmpkfm<p} together yield the claimed value of $k_{\rm s}$. This completes the proof of the theorem.
\end{proof}

In $\mathbb{Z}_{p^k}^{mn}$, let $\bx\in (p)$ denote that $\bx\equiv\bo\pmod{p}$. In such cases, it follows from \eqref{eq:3:ma=} that $\mathbf{A}^{i}(\bx)\in (p)$ when $\bx\in (p)$. Since map~\eqref{eq:3:ma=} is a permutation, the elements in $\{\mathbf{A}^i({\bx})\}_{i=0}^{T_k}$ form a cycle in the functional graph $G_{p^k}$. 
This establishes a distinct separation between states belonging to 
$\{\bx\ \vert\ \bx\in(p)\}$ and those outside it. Consequently, the analysis of the period distribution of map~\eqref{eq:3:ma=} bifurcated based on whether $\bx\in (p)$.
Lemma~\ref{le:3:x(p)iso} discloses that for map~\eqref{eq:3:ma=}, the subgraph of $G_{p^{k}}$ composed of all elements satisfying $\bx\in (p)$ in $\Zp{k}$ is isomorphic to $G_{p^{k-1}}$.

\begin{lemma}\label{le:3:x(p)iso}
The cycle $\{\mathbf{A}^i({\bx})\}_{i=0}^{T_k}$ in $\Zp{k-1}$ 
is isomorphic to the cycle $\{\mathbf{A}^i(p\cdot{\bx})\}_{i=0}^{T_k}$ in $\Zp{k}$ with respect to their corresponding operators.
\end{lemma}
\begin{proof}
Given any $\bx\in\Zp{k-1}^n$, define a multiplication operation $\circ$ for any two elements
$g_1$ and $g_2$ of the set 
$\{\mathbf{A}^i({\bx})\}_{i=0}^{T_k}$:
$g_1\circ g_2=\mathbf{A}^{i_1+i_2}({\bx})\bmod p^{k-1}$,
where $g_1=\mathbf{A}^{i_1}({\bx})$ and $g_2=\mathbf{A}^{i_2}({\bx})$.
Under the operation, the cycle in $\Zp{k-1}^n$ composes a cyclic group of order $T_k$, with the group law
$\mathbf{A}^{i_1}(\bx) \circ \mathbf{A}^{i_2}(\bx)=\mathbf{A}^{i_1+i_2}(\bx)$,
identity $\bx$, and inverses $(\mathbf{A}^i(\bx))^{-1}=\mathbf{A}^{T_k-i}(\bx)\bmod p^{k-1}$.
Similarly, the cycle $\{\mathbf{A}^i(p\cdot{\bx})\}_{i=0}^{T_k}$ in $\Zp{k}$ composes a cyclic group of the same
order under the analogous operation (mod $p^k$).
Define a map
$$
h: \Zp{k-1}^n \;\longrightarrow\; \Zp{k}^n \quad h(\bx)=p\bx.
$$
Since multiplication by $p$ on $\Zp{k}^n$ is injective on residues modulo $p^{k-1}$, $h$ is a bijection from $\Zp{k-1}^n$ onto the subgroup $p\Zp{k}^n$.
Moreover, $\mathbf{A}$ commutes with $h$:
$\mathbf{A}(h(\bx)=\mathbf{A}(p\bx)=\bM(p\bx)=p(\bM \bx)
=h(\mathbf{A}(\bx))$. 
By induction, $h(\mathbf{A}^i(\bx))=\mathbf{A}^i(h(\bx))$ for all $i$.
Thus $h$ restricts to a bijection between the two cycles that intertwines their group operations, establishing a group isomorphism. This completes the proof.
\end{proof}

When $\bx\not\in(p)$, the minimal polynomial of the sequence 
$\{\mathbf{A}^i(\bx)\}_{i\geq0}$ differs from the case where $\bx\in(p)$ previously analyzed.
The evolution of the period of $\{\mathbf{A}^i(\bx)\}_{i\geq0}$ over $\Zp{k}^{mn}$ 
changes as $k$ increases. The analysis of the period is different from the period of the mapping $\mathbf{A}$, because the minimal polynomial of the sequence also changes as $k$ increases. 
Lemma~\ref{le:3A:P_kfp} establishes that, beyond a threshold $k_{\rm s}(\bx_0)$, the period of the sequence in $\mathbb{R}_{k+1}[t]$ increases by a factor of $p$ relative to $T_k(\bx_)$.

\begin{lemma}\label{le:3:xe+1}
For any $\bx_0\in\Zp{k}^{mn}$, one has
 \begin{equation}\label{eq:3:Te+1y}
     T_{k+1}(\bx_0)=
     \begin{cases}
         T_k(\bx_0)         & \text{if } k\leq k_{\rm s}(\bx_0); \\
          p\cdot T_k(\bx_0) & \text{if } k> k_{\rm s}(\bx_0), 
     \end{cases}
 \end{equation}
where
\begin{equation}\label{eq:3a:ksx0}
     k_{\rm s}(\bx)=\max\{k\mid T_k(\bx)=T_1(\bx)\}.
\end{equation}
\end{lemma}
\begin{proof}
Given an initial state $\bx_0$, there is $\bM^{T_1(\bx_0)}(\mathbf{x_0})-\mathbf{x_0}=\bo$ in $\mathbb{Z}_p$.
When $k\leq k_{\rm s}(\bx_0)$, one obtains $T_k(\bx_0)=T_1(\bx_0)$. One can prove 
$t^{p^{k-k_{\rm s}(\bx_0)}T_1(\bx_0)}-1$ is a null polynomial of the sequence in $\mathbb{Z}_{p^k}$ and not a null polynomial in $\mathbb{Z}_{p^{k+1}}$ by mathematical induction on $k$. When $k = k_{\rm s}(\bx_0)$, the statement holds. Assume that the statement holds for $k\leq k_{\rm s}(\bx_0)+s$, there is $v(t)\not\in(p)$ such that $t^{p^sT_1(\bx_0)}-1= p^{k}v(t)$ in $\mathbb{Z}_{p^{k+1}}$. 
Similar to \eqref{eq:3A:qt} and \eqref{eq:3A:t-1rt}, one has $t^{p^sT_1(\bx_0)}-1- p^{k}v(t)\mid t^{p^{s+1}T_1(\bx_0)}-1+p^{k+1}v(t)$, which means that the statement holds for $k=k_{\rm s}(\bx_0)+s+1$. This completes the proof of the statement. When $k> k_{\rm s}(\bx_0)$,
\begin{equation}
    \begin{cases}
        (\bM^{p^{k-k_{\rm s}(\bx_0)}T_1(\bx_0)}-\mathbf{I})(\bx_0)  & \neq \bo;\\
        (\bM^{p^{k-k_{\rm s}(\bx_0)+1}T_1(\bx_0)}-\mathbf{I})(\bx_0) &= \bo
    \end{cases}
\end{equation}
in $\mathbb{Z}_{p^{k+1}}$, namely $T_{k+1}(\bx_0)=pT_{k}(\bx_0)$. This lemma holds.
\end{proof}

Since $\bM$ is a companion matrix, its eigenvalues correspond precisely to the roots $\{\alpha_i\}_{i=1}^m$ of the minimal polynomial $f(t)$. Suppose $t\nmid f(t)$, in $\mathbb{R}_1$, since $\mathbb{R}_1$ is a field and $f(t)$ can be factored, one has the Jordan canonical form $\bJ_1$ of $\bM$ and there is an invertible matrix $\bQ_1$ such that 
\begin{equation}\label{eq:3A:M-PDP}
\bM=\bQ_1\bJ_1\bQ_1^{-1}.
\end{equation}
Then $\bQ_1$ is constructed by the eigenvectors $\{\bv_i\}_{i=1}^m$, where
$\bv_i$ satisfying 
\begin{equation}\label{eq:3A:eignvec}
    (\bM-\alpha_i\mathbf{I})\bv_i=\bo.
\end{equation}
Since $\bQ$ is an invertible matrix, one has 
\begin{equation}\label{eq:3A:detP1}
    \det(\bQ_1)\not\in(p)
\end{equation}
In $\mathbb{R}_k$, similar to \eqref{eq:3A:eignvec}, one has $\bv_{i,k}$ such that 
\begin{equation*}\label{eq:3A:eignveck}
    (\bM-\alpha_{i,k}\mathbf{I})\bv_{i,k}=\bo.
\end{equation*}
Since $\alpha_{i,k}-\alpha_i\in(p)$, one has 
\begin{equation}\label{eq:3A:vk-v}
    \bv_{i,k}-\bv_i\in(p).
\end{equation}
Let $\bQ_k$ be constructed by the eigenvectors $\{\bv_{i,k}\}_{i=1}^m$.
It follows from \eqref{eq:3A:detP1} and \eqref{eq:3A:vk-v} that $\det(\bQ_k)\not\in(p)$. Hence, similar to \eqref{eq:3A:M-PDP}, one has 
\begin{equation*}
\bM=\bQ_k\bJ_k\bQ_k^{-1}.
\end{equation*}
Since $f(t)=\prod_{i=1}^m(t-\alpha_{i,k})$ in $\mathbb{R}_k$, $\{\alpha_{i,k}\}_{i=1}^m$ are eigenvalues of $\bM$, then $\{\alpha_{i,k}\}_{i=1}^m$ are the diagonal elements of the $\bJ_k$.
Based on the above analysis, the variation in the sequence period length in $\mathbb{Z}_{p^k}$ as $k$ increases is illustrated by Lemma~\ref{le:3:xe+1}.

\begin{lemma}\label{le:3A:ksx0}
Equation~\eqref{eq:3:Te+1y} holds with
\[
k_{\rm s}(\bx_0) = \min_{1\leq j \leq m_h} k(\alpha_j),
\]
where
\[
k(\alpha_j) = \max\{k \mid \ord_k(\alpha_{j,k}) = \ord_1(\alpha_j)\},
\]
$\{\alpha_j\}_{1\leq j\leq m_h}$ are the roots of the minimal polynomial $h(t)$ of the sequence $\{\bA^i(\bx_0)\}_{i\geq 0}$ in $\mathbb{R}_1$, and $m_h = \deg h(t)$.
\end{lemma}
\begin{proof}
In $\mathbb{R}_{1}$, since $\mathbb{Z}_p$ is a finite field, for the sequence generated by $\bx_0$, $h(t)$ as a unique minimal polynomial can divide all null polynomials, that is, $h(t)\mid f(t)$. Hence, there are all roots $\{\alpha_{j}\}_{1\leq j\leq m_h}$ of $h(t)$ such that $\alpha_{j}^{T_{1}(\bx_0)}-1=0$. 
Referring to
\begin{equation}\label{eq:3A:MTMPT}
        \bM^{T_{1}(\bx_0)}\bx_0\equiv \bx_0 \pmod{p}, 
\end{equation}    
and 
\[
\bM^{s\cdot T_{k}(\bx_0)}-\mathbf{I}=\bQ_{k}(\bJ_{k}^{s\cdot T_{k}(\bx_0)}-\mathbf{I})\bQ_{k}^{-1},
\]
one obtains
\begin{equation}\label{eq:3A:J1T}
        (\bJ_{1}^{T_{1}(\bx_0)}-\mathbf{I})\by_0= \bo 
\end{equation}
where $\by_0=\bQ_{1}^{-1}\bx_0$. 

One can prove the statement that there exists a vector $\mathbf{y}_{0,s}$ with no zero elements such that 
\begin{equation}\label{eq:J1s}
    (\bJ_{1,s}^{T_1(\bx_0)}-\mathbf{I})\by_{0,s}=\bo,
\end{equation}
where $\mathbf{J}_{1,s}$ is a submatrix of $\mathbf{J}_1$ corresponding to $f_s(t)$ and $f_s(t)$ is the null polynomial of both $\{\mathbf{J}_{1,s}^i \mathbf{y}_{0,s}\}_{i\geq 0}$ and the $\{\bA^i(\bx_0)\}_{i\geq 0}$ by mathematical induction on $s$. When $s = 1$, let $\mathbf{y}_{0,1}=\mathbf{y}_{0}$ and $\mathbf{J}_{1,s}=\mathbf{J}_{1,1}$, it follows from ~\eqref{eq:3A:J1T} that the statement holds. Assume that the statement holds for $s=l$.
If the $i$-th element of $\by_{0,l}$ is $0$, one has the $i$-th element of $(\bJ_{1,l}^{T_1(\bx_0)}-\mathbf{I})\by_{0,l}$ is $0$. Note that the null polynomial of the sequence $\{\bA^i(\bx_0)\}_{i\geq 0}$ is also the null polynomial of the sequence of $\{\bJ_{1,l}^i(\by_{0,l})\}_{i\geq 0}$ with deleting the $i$-th element. By matrix transformation, there is a Jordan canonical form $\bJ_{1,l+1}$ of $\frac{f_s(t)}{t-\alpha_{i}}$ which is the matrix deleting the $i$-th row and column of $\bJ_{1,l}$. Then 
\[
(\bJ_{1,l+1}^{T_1(\bx_0)}-\mathbf{I})\by_{0, l+1}=\bo,
\]
where $\by_{0, l+1}$ is $\by_{0,l}$ with deleting the $i$-th element. 
Hence, $\frac{f_s(t)}{t-\alpha_{i}}$ is the null polynomial of 
$\{\bJ_{1,l+1}^i\by_{0,l+1}\}_{i\geq 0}$ and
$\{\bA^i(\bx_0)\}_{i\geq 0}$. This statement holds.

Since $h(t)$ is the minimal polynomial of $\{\bA^i(\bx_0)\}_{i\geq 0}$, there is $s$ such that $h(t)\mid f_s(t)$ and any element of $\by_{0,s}$ is not $0$.
Let $\alpha_j$ is a root of $f(t)$ in $\mathbb{R}_1$, 
there is the $i_j$-th row of $\mathbf{J}_{1,s}$ where only the diagonal element is $\alpha_j$.
Since ~\eqref{eq:J1s}, in $\mathbb{R}_{k(\alpha_{i_j})+1}$, there is $\alpha_{i_j, k}^{T_1(\bx)}-1= c p^{k(\alpha_{i_j})}$.
It follows from ~\eqref{eq:J1s} that $(\bJ_{k,s}^{T_1(\bx)}-\mathbf{I})\by_{0,s}\neq \bo$. Then $(\bJ_{k}^{T_1(\bx)}-\mathbf{I})\by_{0}\neq \bo$, namely $(\bM^{T_1(\bx)}-\mathbf{I})\bx_{0}\neq \bo$.
 Since ~\eqref{eq:3a:ksx0}, one has $k_{\rm s}(\bx_0)=\min k(\alpha_{i_j})$.
\end{proof}

\begin{lemma}\label{le:3:xe+1c}
For any $\bx_0\in\mathbb{Z}_{p^{k'}}^{\times}$, one has
    \begin{equation}\label{eq:3A:Tx0+c}
        T_{k}(\bx_0+ p^{k'}\cdot{\bf c})=\lcm( T_{k}(\bx_0), T_{k-k'}({\bf c})), 
    \end{equation}
where ${\bf c}\in\mathbb{Z}_{p^{k-k'}}^{mn}$.
\end{lemma}
\begin{proof}
Let $l=\lcm(T_{k}(\bx_0), T_{k-k'}({\bf c}))$. 
In $\mathbb{Z}_{p^{k}}$, since $\bM^{ T_{k}(\bx_0)}\bx_0=\bx_0$, one obtains 
$\bM^{l}(\bx_0+ p^{k'}\cdot{\bf c})=\bx_0+ p^{k'}\cdot{\bf c}$.
Hence,
\[
T_{k}(\bx_0+ p^{k'}\cdot{\bf c})\mid l.
\]
Since the sequence generated with initial state $\bx_0+ p^{k'}\cdot{\bf c}$ in $\mathbb{Z}_{p^{k}}$ corresponds to the sequence generated with initial state $\bx_0$ in $\mathbb{Z}_{p^{k-1}}$, one has 
\begin{equation}\label{eq:3A:Tk|Tk+1}
    T_{k-1}(\bx_0)\mid T_{k}(\bx_0+ p^{k'}\cdot{\bf c}).
\end{equation}
There are two cases of this proof whether $T_{k}(\bx_0)=T_{k+1}(\bx_0)$ or not.
\begin{itemize}
    \item $T_{k-1}(\bx_0)=T_{k}(\bx_0)$: one has $T_{k}(\bx_0+ p^{k'}\cdot{\bf c})=s\cdot T_{k-1}(\bx_0)$, where $s$ is an integer. Then 
\begin{equation*}
    \bM^{s\cdot T_{k}(\bx_0)}(\bx_0+ p^{k'}\cdot{\bf c})=\bx_0+p^{k'}\cdot\bM^{s\cdot T_{k}(\bx_0)}{\bf c}=\bx_0+p^{k'}\cdot {\bf c}.
\end{equation*}
namely $p^{k'}\cdot\bM^{s\cdot T_{k}(\bx_0)}{\bf c}={\bf c}$. Hence $T_{k-k'}(\mathbf{c})\mid s\cdot T_{k+1}(\bx_0)$, that is, $T_{k+1}(\bx_0+ p^{k'}\cdot{\bf c})=l$.
\item $T_{k-1}(\bx_0)\neq T_{k}(\bx_0)$, it follows from \eqref{eq:3:Te+1y} that $T_{k}(\bx_0)=p \cdot T_{k-1}(\bx_0)$. 
In $\mathbb{R}_{k}$, referring to
\begin{equation}\label{eq:3A:MTMPT}
    \left\{\begin{split}
        \bM^{T_{k-1}(\bx_0)}\bx_0&\equiv \bx_0 \pmod{p^{k-1}}; \\
        \bM^{T_{k-1}(\bx_0)}\bx_0&\not\equiv \bx_0 \pmod{p^{k}}
    \end{split}
    \right.
\end{equation}
and 
\[
    \bM^{s\cdot T_{k-1}(\bx_0)}-\mathbf{I}=\bQ_{k}(\bJ_{k}^{s\cdot T_{k-1}(\bx_0)}-\mathbf{I})\bQ_{k}^{-1},
\]
one obtains
\begin{equation}\label{eq:3A:Dy}
\left\{
\begin{split}
        (\bJ_{k-1}^{T_{k-1}(\bx_0)}-\mathbf{I})\by_0&\in (p^{k-1}),\\
         (\bJ_{k-1}^{T_{k-1}(\bx_0)}-\mathbf{I})\by_0&\not\equiv \bo 
\end{split}
 \right.
\end{equation}
where $\by_0=\bQ_{k}^{-1}\bx_0$. Note that $\{\alpha_{i,k}\}_{i=1}^m$ are the diagonal elements of the $\bJ_k$. Assume $\alpha_{i,k}^{T_{k-1}(\bx_0)}-1\neq 0\pmod{p^{k-1}}$ for any $i$. 
In $\mathbb{R}_k$, $\det(\bJ_k^{T_{k-1}(\bx_0))}-\mathbf{I})\neq 0$. It follows from the above equation that $\by_0=\bo$. Since $\bQ_{k}$ is an invertible matrix, one has $\bx_0\equiv \bo$, which is contradicted by $\bx_0\in\mathbb{Z}_{p^{k'}}^{\times}$. Hence, there is $\alpha_{i_i+l,k}$ such that  
\begin{equation}\label{eq:3A:alphatk-1}
 \left\{\begin{split}
     \alpha_{i_1,k}^{T_k(\bx_0)}-1 & \in(p^{k-1});\\
 \alpha_{i_1,k}^{T_k(\bx_0)}-1     & \not\equiv 0\pmod{p^{k}},
 \end{split}
 \right.
 \end{equation}
 Let $s$ is algebraic multiplicity of the value $\{\alpha_{i_i,k}\}$.
 For the $s\times s$ Jordan block in $\bJ_k$ corresponding $\alpha_{i_1,k}$, there is $j$-th row of $\bJ$ corresponding to its bottom row has only value $\alpha_{i_1,k}$ on the diagonal and zeros elsewhere. 
Referring to \eqref{eq:3A:Dy}, one has $y_{j}\neq 0\pmod{p}$.
Assume 
\[
T_{k}(\bx_0)\nmid T_{k}(\bx_0+ p^{k'}\cdot{\bf c}).
\]
Since relation~\eqref{eq:3A:Tk|Tk+1}, one has 
\begin{equation*}
        \bM^{s\cdot T_{k-1}(\bx_0)}(\bx_0+p^{k'}\cdot \mathbf{c})=\bx_0+p^{k'}\cdot \mathbf{c}
\end{equation*}
namely
\[
 (\bM^{s\cdot T_{k-1}(\bx_0)}-\mathbf{I})\bx_0=-(\bM^{s\cdot T_{k-1}(\bx_0)}-\mathbf{I})(p^{k'}\cdot \mathbf{c}).
\]
According to \eqref{eq:3A:Dy}, one has 
\[
(\bJ_{k}^{s\cdot T_{k-1}(\bx_0)}-\mathbf{I})\by_0= p^{k'}\cdot (\bJ_{k}^{s\cdot T_{k-1}(\bx_0)}-\mathbf{I})\mathbf{b},
\]
where $\mathbf{b}=\bQ_{k}^{-1}\mathbf{c}$. Then 
\[
    (\alpha_{i_1,k}^{s\cdot T_{k-1}(\bx_0)}-1)y_{i_1}=0,
\]
which is a contradiction.
Hence, 
\[
T_{k}(\bx_0)\mid T_{k}(\bx_0+ p^{k'}\cdot{\bf c}).
\]
It follows from
\[
\begin{split}
    &\bM^{ T_{k}(\bx_0+ p^{k'}\cdot{\bf c})}(\bx_0+ p^{k'}\cdot{\bf c})\\
    &=\bx_0+p^{k-1}\bM^{T_{k}(\bx_0+ p^{k'}\cdot{\bf c})}{\bf c}\\
&=\bx_0+{\bf c}
\end{split}
\]
that
\[
\bM^{T_{k}(\bx_0+ p^{k'}\cdot{\bf c})}{\bf c}={\bf c},
\]
then $T_{k-k'}(\mathbf{c})\mid T_{k}(\bx_0+ p^{k'}\cdot{\bf c})$.
So, $T_{k}(\bx_0+ p^{k'}\cdot{\bf c})=\lcm(T_{k}(\bx_0), T_{k-k'}({\bf c}))$. 
\end{itemize}
Combining the above analysis, this lemma holds.
\end{proof}

Given a state $\bx\in\mathbb{Z}_{p^{k^*}}^{mn}$, there exists a cycle containing $\bx$ in $G_{p^{k^*}}$, and all elements of this cycle form a set denoted by $C_{k}({\bx})=\{ \hat{\bx} \vvert \hat{\bx}=\mathbf{A}^i(\bx), \bx\in{\Zp{k}}, i\in \mathbf{Z}\}$. 
As $k$ increases, for any element $\bx'$ corresponds to an element ${\bx}$ that composes a set $N_{k, \bx}=\{\bx'\vert \bx'\equiv\bx\pmod{p^{k^*}}\}$. 
Hence, $C_{k}(\bx')$ corresponds to a cycle $C_{k^*}(\bx)$ in $G_{p^{k^*}}$. 
Let
\[
S_{k, \bx} = \bigcap_{\bx'\in N_{k, \bx}} C_{k}(\bx')
\]
for $\bx\in\mathbb{Z}^{mn}_p$ indicates that all elements in this set form several cycles corresponding to the cycle $C_{k^*, \bx}$. 
Furthermore, let 
\begin{equation}\label{eq:3A:NTS}
    N(T, S_{k, {\bx}})=|\{C_{k}(\bx')\vvert C_{k}(\bx')\subseteq S_{k, {\bx}}, |C_{k}(\bx')|=T\}|
\end{equation}
be the number of cycles of length $T$ in the set $S_{k, {\bx}}$.

\begin{lemma}\label{le:3A:N=a}
   Given $T^*>P_1(f)$, for any $\bx_0$ such that $T_k(\bx_0)=T^*$, 
   there is $S_{k+1}(\bx_0)=\{\bx \vert (\bx \bmod p^k)\in C_k(\bx_0), \bx\in\mathbb{Z}_{p^{k+1}}^{mn}\}$
   and 
   \[
   N(T, S_{k+1}(\bx_0))=\frac{a(T^*, T)T^*}{T},
   \]
   where
       \begin{equation}\label{eq:3a:aTT}
        a(T^*, T_1)=|\{\mathbf{c}\vvert \mathbf{c}\in\mathbb{Z}^{mn}_p,\lcm(pT^*, T_1(\mathbf{c}))=T\}|
    \end{equation}
    is a constant depending on $T^*$ and $T$.
\end{lemma}
\begin{proof}
    From Lemma~\ref{le:3:xe+1c}, one obtains $|C_{k+1}(\bx_0+p^k \mathbf{c})|=T$ if and only if $\lcm(pT^*, T_1(\mathbf{c}))=T$. 
    Hence 
    \begin{multline*}
    |\{\bx\vvert \bx\in S_{k+1}(\bx_0), |C_{k+1}(\bx)|=T\}|=a(T^*, T)\cdot T^*.
    \end{multline*}
    It follows that 
    \begin{equation*}
            N(T, S_{k+1}(\bx_0))=\frac{a(T^*, T)T^*}{T}.
    \end{equation*}
    Then this lemma holds.
\end{proof}

Given an initial state $\bx_0$, if $k_{\rm s}(\bx_0)\rightarrow \infty$, one has $T_k(\bx_0)=T_1(\bx_0)$, that is, $\alpha_{i,k}^{T_1(\alpha_1)}= 1$.
Since any root of $f(t)$ has Cauchy's Bound, when $k$ is large enough, there is $\alpha_{i,k+l}=\hat{\alpha}_{i}$ for any $l$, where $\hat{\alpha}_{i}$ is a root of $f(t)$ in $\mathbb{N}$. If $\hat{\alpha}_{i}\not\in\{1,-1\}$, $\alpha_{i,k}^s\neq 1$ in $\mathbb{N}$ for any $s$. Then there is a $k'$ that $\alpha_{i,k}^{T_1(\alpha_1)}\neq 1$ in $\mathbb{Z}_{p^{k'}}$. So if $f(1)\neq 0$ and $f(-1)\neq 0$, referring to lemma~\ref{le:3A:ksx0}, for any $\bx_0$, $k_{\rm s}(\bx)$ is finite.

\begin{theorem}\label{th:3A:Nkv}
There exists 
\begin{equation}\label{eq:the:Nkv}
    \mathbf{N}_{k+l}(v+l)= \bD^l\cdot \mathbf{N}_{k}(v)
\end{equation}
for any $v>\bnu(P_1(f))$, where ${\bf N}_k(v)$ is a column vector of size $r\times 1$ with the $i$-th component $N_k(T_{k, i}(v), \mathbb{Z}^{mn}_{p^k})$, 
$\{T_{k, i}(v)\}_{i=1}^r= \{T\vvert \bnu(T)=v, T\mid P_k(f)\}$ and $T_{k, i}(v)<T_{k, i+1}(v)$, 
$\bD$ is a $r\times r$ lower triangular matrix, and $l\in \mathbb{N}^+$.
\end{theorem}
\begin{proof}
Let $T_{1,i}=T_{1,i}(0)$, since the definition of $T_{k,i}(v)$, one has $T_{k,i}(v)=p^v T_{1,i}$.
For any $i\in\{1,2,\cdots, r\}$,  
there is a set 
\begin{equation}\label{eq:3A:SkTki}
    S_k(T_{k,i}(v))=S_k(p^vT_{1,i})=\{\bx \vvert \bx\in\mathbb{Z}^{mn}_{p^k} ,|C_k(\bx)|=p^vT_{1,i}\}.
\end{equation}
    For any $\bx\in \bigsqcup_{j=1}^r S_k(p^vT_{1,i})$, it follows from Lemma~\ref{le:3:xe+1c} that $\bnu(T_{k}(\bx))=\bnu(T_{k,i}(v))=v$, where $\bigsqcup$ is the  disjoint union operator. Combining Lemma~\ref{le:3:xe+1} with Lemma~\ref{le:3:xe+1c}, for any $\bx'\in\mathbb{Z}^{mn}_{p^{k+1}}$ and $\bx'\equiv\bx_0\pmod{p^k}$, one has $\bnu(T_{k+1}(\bx'))=v+1$. Then $\bx' \in \bigsqcup_{j=1}^r S_{k+1}(p^{v+1}{T}_{1,j})$. Similarly, for any $\bx'\in\bigsqcup_{j=1}^r S_{k+1}(p^{v+1}{T}_{1,j})$, there is $\bx=\bx'\pmod{p^k}$ such that $\bx\in\bigsqcup_{j=1}^r S_k(p^vT_{1,i})$, that is, 
    \begin{equation}\label{eq:3A:Sk+1capcap}
            \bigsqcup_{i=1}^r S_{k+1}(p^{v+1}T_{1,i})=\bigcup_{i=1}^r S'(p^vT_{1,i}),
    \end{equation}
    where $S'(p^vT_{1,i})=\{\bx\vvert C_k(\bx)=p^vT_{1,i}, \bx\in\mathbb{Z}_{p^{k+1}}^{mn}\}$.
    
    From Lemma~\ref{le:3A:N=a}, for any $\bx_0$ satisfying $T_k(\bx_0)=p^vT_{1,i}$, there is 
    \begin{equation}\label{eq:3a:N=bij}
    N(p^{v+1}T_{1,j}, S_{k+1}(\bx_0))=b_{i,j},
    \end{equation}
where $b_{i,j}=\frac{a_{ij}}{p}$ and
    \begin{equation}\label{eq:3a:aij}
    \begin{split}
        a_{ij}&=a(p^kT_{1,i}, p^{v+1}T_{1,j})\\
        &=\left|\left\{\mathbf{c} \Bigg \vert  \mathbf{c}\in\mathbb{Z}^{mn}_p,\lcm\left(T_{1,i}, \frac{T_1(\mathbf{c})}{p^{\bnu(T_1(\mathbf{c}))}}\right)=T_{1,j}\right\}\right|.
    \end{split}
\end{equation}
Since $S_{k+1}(\bx_0)\in S'(p^vT_{1,i})$ for any $\bx_0$ satisfying $T_k(\bx_0)=p^vT_{1,i}$,
one has
    \[
    N\left(p^{v+1}T_{1,j}, S'(p^vT_{1,i})\right)=p\cdot a_{ij} N(p^vT_{1,i}, \mathbb{Z}_p^k).
    \]
Referring to ~\eqref{eq:3A:Sk+1capcap}, one obtains 
    \begin{equation*}
    \begin{split}
        N\left(p^{v+1}T_{1,j}, \mathbb{Z}_{p^{k+1}}^{mn}\right)&=N\left(p^{v+1}T_{1,j}, \bigcup_{i=1}^r S'(p^vT_{1,i})\right)\\
        &=\frac{1}{p}\sum_{i=1}^r a_{ij} N(p^vT_{1,i}, \mathbb{Z}_p^k).
    \end{split}
\end{equation*}
Hence
${\bf N}_{k+1}(v+1)=\mathbf{D}{\bf N}_{k}(v)$.
When $i<j$, it follows from $T_{1,i}<T_{1,j}$ that $T_{1,j}\nmid T_{1,i}$, then $a_j=0$ and $b_{i,j}=0$.
Hence
\[
\mathbf{D}=
\begin{bmatrix}
    b_{1,1} & 0& \cdots,&0\\
    b_{1,2} & b_{2,2}& \cdots,&0\\
    \vdots  & \vdots & \ddots &\vdots\\
    b_{1,r} & b_{2,r}& \cdots,&b_{r,r}
\end{bmatrix}
\].
It follows from ~\eqref{eq:3a:aij} that $b_{i,j}$ is constant depending on $i$ and $j$.
So ~\eqref{eq:the:Nkv} holds.
\end{proof}


\begin{theorem}\label{th:3A:fe*}
If $f(1)\neq 0$ and $f(-1)\neq 0$, then there exists a threshold $\hat{k}_{\rm s}=\max\{k_{\rm s}(\bx) \vvert \bx\in\Zp{}^{mn}, \bx\neq \bo\}$ such that
for $k\ge \hat{k}_{\rm s}$, 
    \begin{equation}\label{prop:Ntce2}
	  N_{T, \hat{k}_{\rm s}+\bnu(T)+l}=N_{T, \hat{k}_{\rm s}+\bnu(T)},
	\end{equation}
where $l\in \mathbb{N}^+$.
\end{theorem}
\begin{proof}
If $f(1)\neq 0$ and $f(-1)\neq 0$, there is a finite integer $\hat{k}_{\rm s}$.
When $k>\hat{k}_{\rm s}$, for any $\bx\in\{\bx \vvert \bx\in\Zp{k}^{mn},\bx\not\equiv\bo\pmod{p}\}$,
there is $\bnu(T_k(\bx))\geq {k-\hat{k}_{\rm s}}$. It follows from Lemma~\ref{le:3:x(p)iso} that $N_{T, k}=N_{T, k-1}$ for any $T$ satisfying $\bnu(T)<{k-\hat{k}_{\rm s}}$. Further, it can be obtained that 
\[
N_{T, k+l}=N_{T, k-1}
\]
for any $T$ satisfying $\bnu(T)<{k-\hat{k}_{\rm s}}$.
Since Lemma~\ref{le:3:xe+1} and the definition of $\hat{k}_{\rm s}$, when $k=\hat{k}_{\rm s}+\bnu(T)+1$, there is $\bnu(T)<k-\hat{k}_{\rm s}$. Hence 
~\eqref{prop:Ntce2} holds.
\end{proof}

\newcommand{\minitab}[2][l]{\begin{tabular}{#1}#2\end{tabular}}
\setlength\tabcolsep{2pt} 
\addtolength{\abovecaptionskip}{0pt}
\renewcommand{\arraystretch}{1.2}

\begin{table*}[!htb]
\caption{Conditions on $(a, b)$ and the corresponding number of valid cases $N_T$ for each possible period $T$ of Cat map~\eqref{eq:ArnoldInteger} over $\mathbb{Z}_{p^k}$.}
	\centering 
	\begin{tabular}{*{2}{c|}c} 
		\hline 
		$T$  & Condition &  $N_T$ \\ \hline		
		{1  }   & $(a,b)=(0,0)$ &   1   \\ \hline
       \multirow{2}{*}{$p^k$}    &$ab\equiv0\pmod{p}, a\not\equiv0\pmod{p}$&   $p^{2e-2}(p-1)$            \\ \cline{2-3}
       &$ab\equiv0\pmod{p}, b\not\equiv0\pmod{p}$&   $p^{2e-2}(p-1)$            \\
        \hline		
        {$2 p^k$}      & $ab\equiv p-4\pmod{p}$& $p^{2e-2}(p-1)$   \\ \hline
		\multirow{2}{*}{\minitab[c]{$p^i$,\\ $i\in\{1, 2, \ldots, e-1\}$} }	
        & $\min(\bnu(a), \bnu(b))=p^{k-i}$& $2(p-1)(p^{2i-1})$\\ \cline{2-3}
		          & $ab=0, \max(\bnu(a),\bnu(b))=p^{k-i}$& $2(p-1)(p^{i-1})$    \\   \hline
		\makecell{\minitab[c]{$k_1\cdot k_2$, \\ $k_1>2$, $k_1 \mid p-1$, $k_2 \mid p^{e-1}$}} & $a,b\in\mathbb{Z}^\times_{p^k}  $&      $\frac{\varphi(k_1 k_2)}{2}(p^k-p^{k-1})$    \\  \hline
		\makecell{\minitab[c]{\minitab[c]{$k_1\cdot k_2$, \\ $k_1>2$, $k_1 \mid p+1$, $k_2 \mid p^{e-1}$}}}		  &$a,b\in\mathbb{Z}^\times_{p^k}  $ & $\frac{\varphi(k_1 k_2)}{2}(p^k-p^{k-1})$     \\ \hline 	
	\end{tabular}
	\label{tab:ParametersCat}
\end{table*}

\subsection{The Functional Graph of Cat Map over $(\Zp{k}, +, \ \cdot\ )$}

In \cite{Chen:cat:TIT2012}, F. Chen et al. enumerates all Cat maps over $(\mathbb{Z}_{p^k}, +, \ \cdot\ )$ by their periods
using the generation function and the Hensel lifting techniques, where $p$ is a prime number larger than $3$.
To make this paper more self‐contained, Table~\ref{tab:ParametersCat} concisely lists the explicit $(a, b)$ pairs that 
yield each attainable period. The paper is unrelated to Cat map’s local cycle structure, as the period of any permutation on a finite domain is the least common multiple of all its cycle lengths.
Adopting the analytical framework developed above, we fill this gap by deriving the exact distribution of cycle lengths for Cat maps over $(\Zp{k}, +, \ \cdot\ )$.

The Cat map over $(\Zp{k}, +, \ \cdot\ )$ can be represented as
\begin{equation}\label{eq:ArnoldInteger}
    f
    \begin{bmatrix}
        x_i\\
        y_i
    \end{bmatrix}
  =
        \begin{bmatrix}
        x_{i+1}\\
        y_{i+1}
    \end{bmatrix}
 ={\bf C}
    \begin{bmatrix}
        x_i\\
        y_i
    \end{bmatrix}
    \bmod {p^k}, 
\end{equation}
where
\[
{\bf C}=
    \begin{bmatrix}
        1 & a\\
        b & 1+a\cdot b
    \end{bmatrix}, 
\]
$x_i, y_i\in \Zp{k}$ and $a, b\in\Zp{k}$.
when $a,b = 0$, one has ${\bf C}={\bf I}$, then the minimal polynomial of the Cat map is $t-1$. 
When $a, b\neq 0$, one has ${\bf C}\neq c\cdot {\bf I}$ for any integer $c$, namely the degree of the minimal polynomial is more than one. Since
\[
{\bf C}^2=(ab+2) {\bf C}-{\bf I}, 
\]
that is, 
\[
    \begin{bmatrix}
        x_{i+2}\\
        y_{i+2}
    \end{bmatrix}
 =(ab+2)
    \begin{bmatrix}
        x_{i+1}\\
        y_{i+1}
    \end{bmatrix}-
    \begin{bmatrix}
        x_i\\
        y_i
    \end{bmatrix}
\]
and $\{(x_i, y_i)^\intercal\}_{i=0}^{\infty}$ is a periodic sequence.
Hence, the minimal polynomial of the Cat map is $f(t)=t^2-(ab+2)t+1$. 
It is similar to \eqref{eq:3:ma=} that
\begin{equation*}
    \mathbf{A}^i(\mathbf{z}_0)=\bM^i
\begin{bmatrix}
    x_0 \\
    y_0 \\
    x_1 \\
    y_1
\end{bmatrix}
=\begin{bmatrix}
    x_i \\
    y_i \\
    x_{i+1} \\
    y_{i+1}
\end{bmatrix}, 
\end{equation*}
where
\[
\bM=
\begin{bmatrix}
    0     & 0     & 1  & 0\\
    0     & 0     & 0  & 1\\
    ab+2  & 0     & -1 & 0\\
    0     & ab+2  & 0  & -1
\end{bmatrix}.
\]

Just like \cite{cqli:Cat:TC22}, let $z_i=x_i+p^e\cdot y_i$ represent the state in the functional graph of the Cat map. Then, the state in the functional graph of the Cat map can be represented by $(z_i, z_{i+1})$.
There is a distinction between $\mathbf{A}(x)$ and the Cat map. For $\mathbf{A}(x)$, the number of states in the functional graph of $\mathbf{A}(x)$ is $p^{4e}$ in $\Zp{k}$. But for the Cat map, because $z_{i+1}$ is unique for a given $z_{i+1}$, the number of state of $z_i$ in $\Zp{k}$ is $p^{2e}$.
Denote $G_e$ as the functional graph of the Cat map with precision $e$. It is isomorphic to the subgraph of the functional graph of $\mathbf{A}(x)$ with precision $e$. As illustrated in Fig.~\ref{Fig:CandA}, for any state $z_0$ in Fig.~\ref{Fig:CandA}, there is a unique $z_1$ as depicted in Fig.~\ref{Fig:CandA}a) and the cycle include state $(z_0, z_1)$ in Fig.~\ref{Fig:CandA}b) is isomorphic to the cycle include $z_i$ in Fig.~\ref{Fig:CandA}a).
Building upon the discourse surrounding $\mathbf{A}(x)$ in Sec.~\ref{subsec:3A}, several properties of the functional graph of the Cat map over $(\Zp{k}, +, \ \cdot\ )$ are proposed.

\begin{figure}[!htb]
    \centering
     \raisebox{-2.1em}{\begin{minipage}{0.8\twofigwidth}
      \centering
     \includegraphics[width=0.8\twofigwidth]{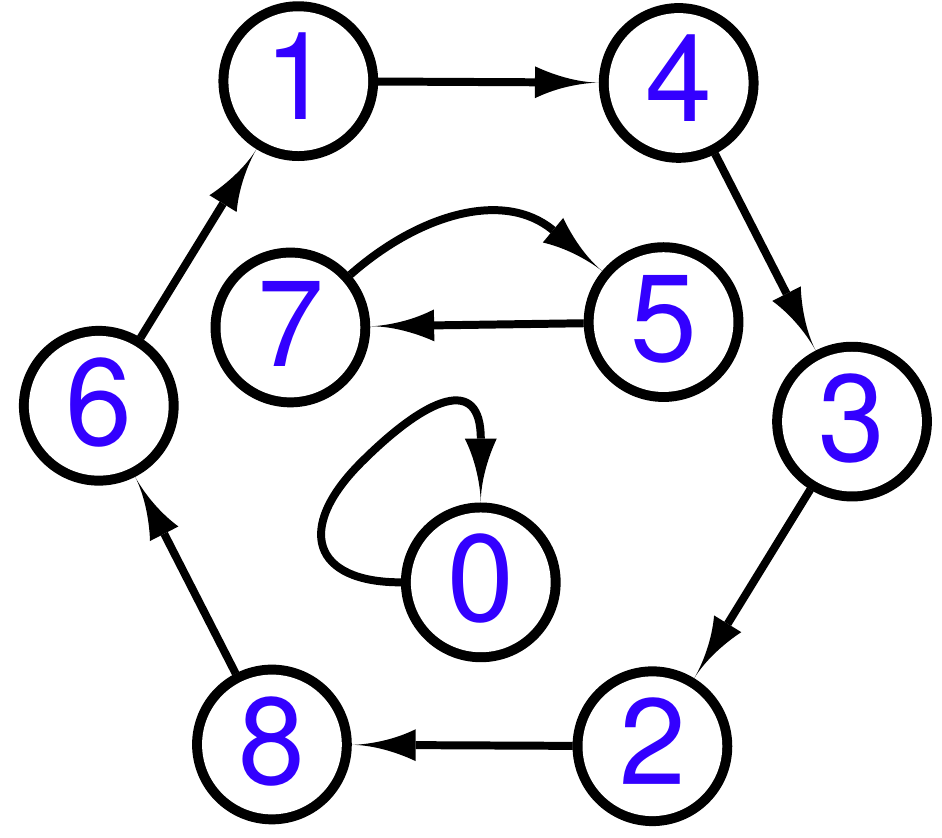}
     a)
     \end{minipage}}
     \begin{minipage}{1.2\twofigwidth}
        \centering
     \includegraphics[width=1.2\twofigwidth]{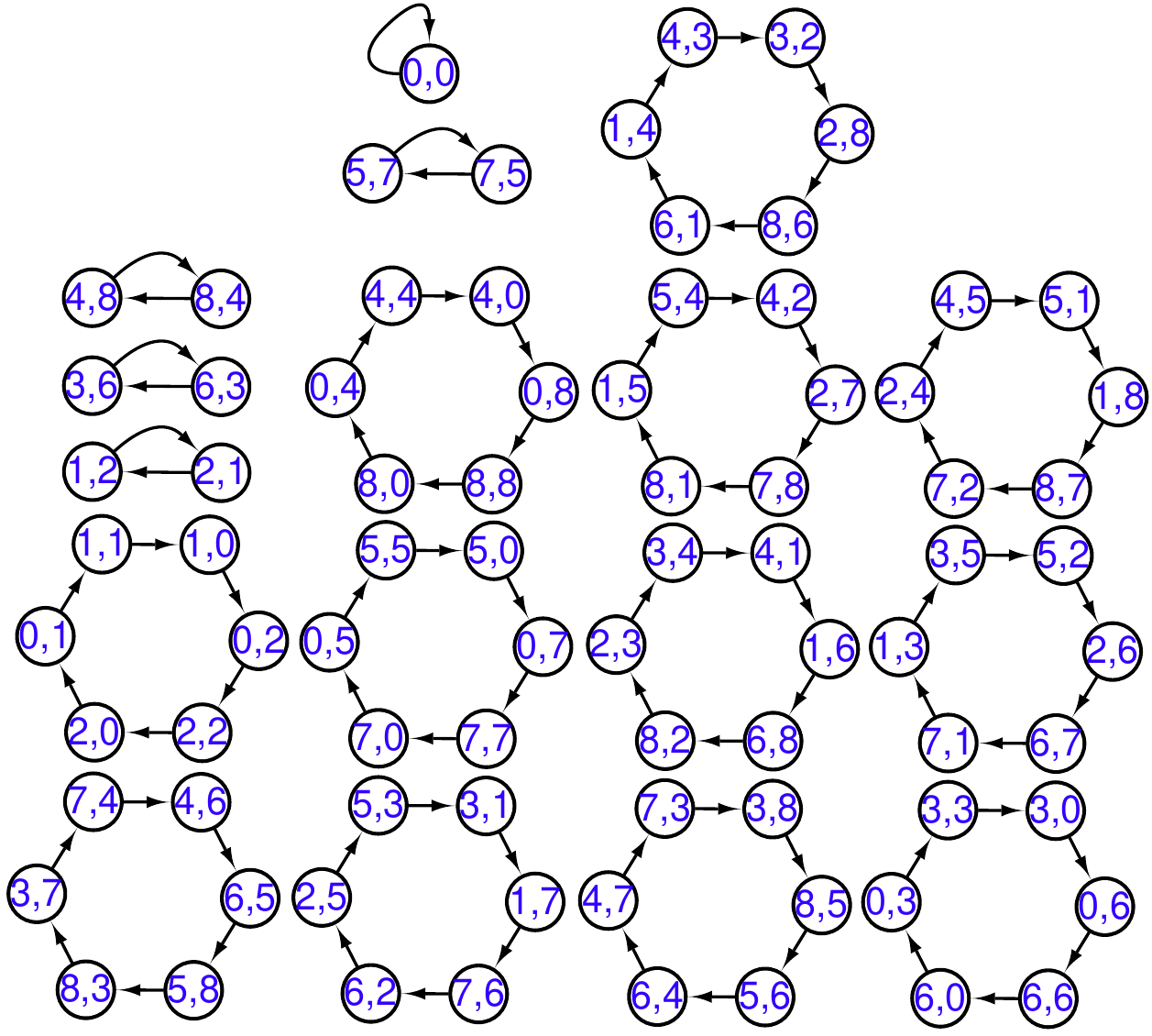}
    b)
     \end{minipage}
\caption{The functional graph of two maps over $\mathbb{Z}_3$: 
a) Cat map $(\ref{eq:ArnoldInteger})$ with parameter $(a, b)=(1, 5)$; 
b) $\mathbf{A}(\mathbf{z})$ with minimal polynomial $f(t)=t^2-7t+1$.}
\label{Fig:CandA}
\end{figure}

\begin{theorem}\label{le:3:cat1}
When $a,b$ are not both zero, the least period of the Cat map in $\Zp{k}$ is ${T}_k$ and there exists a threshold of $k$, 
\begin{equation}\label{eq:Cat:kf}
 k_{\rm s}=
 \begin{cases}
     1             & \mbox{if } \alpha=\alpha^{-1};\\
 f_{\rm k}(\alpha) & \mbox{if } \alpha\neq\alpha^{-1},  
 \end{cases}
\end{equation}
satisfying
\begin{equation}\label{eq:3B:Tcxe=pe*}
    T_{k+1}=p^{k-k_{\rm s}}\cdot T_{k_{\rm s}}
\end{equation}
when $k\geq k_{\rm s}$, 
where $f(t)=t^2-(ab+2)t+1$, $\alpha$ is a root of $f(t)$ in $\mathbb{R}_1$, 
\begin{equation*}
f_{\rm k}(\alpha)=\max\left\{ k \vvert \ord_{k}\left(g^{(k-1)}(\alpha)\right)=\ord_1(\alpha) \right\}, 
\end{equation*}
and
$g^{(n)}(x)$ indicate applying 
\begin{equation*}
g(x)=x-f'(x)^{-1}f(x)
\end{equation*}
over $\mathbb{R}_k$ $n$ times.
\end{theorem}
\begin{proof}
When $a\neq0$ or $b\neq0$, the minimal polynomial is $f(t)=t^2-(ab+2)t+1$.
According to Theorem~\ref{coro:2:e>1b=1}, this theorem holds.
\end{proof}

The Cat map as a kind of $\mathbf{P}(\mathbf{a})$ satisfies some properties:
\begin{itemize}
    \item For $\mathbf{P}(\mathbf{a})$ in $\mathbb{Z}_{p^k}^n$, there is a subset $\mathbb{S}_{p^k}^n$ of $\mathbb{Z}_{p^k}^n$ composed of all states of $\mathbf{P}(\mathbf{a})$. 
    \item As $k$ increases, for any $\mathbf{a}\in\mathbb{S}_{p^{k+1}}^n$, there is $\mathbf{c}$ such that $\mathbf{a}=\mathbf{a}_k+p^k\mathbf{c}$, where $\mathbf{a}_k\in\mathbb{S}_{p^{k}}^{mn}$. 
\end{itemize}
Let $\mathbb{C}^{mn}_p$ be the set composed of all states of $\mathbf{c}$. 
Similar to Lemma~\ref{le:3A:N=a} and Theorem~\ref{th:3A:Nkv}, the period distribution of a given state as $k$ to $k+1$ is disclosed in Corollary~\ref{coro:3b:PaNTSx0} and a property of the period distribution as $k$ increases is disclosed in Corollary~\ref{coro:3b:PaNTS}.
\begin{Corollary}\label{coro:3b:PaNTSx0}
    Given $T^*>P_1(f)$, for any $\mathbf{a}_0$ such that $T_k(\mathbf{a}_0)=T^*$, 
   there is $S_{k+1}(\mathbf{a}_0)=\{\mathbf{a}\vert (\mathbf{a} \bmod p^k)\in C_k(\mathbf{a}_0), \mathbf{a} \in\mathbb{S}_{p^{k+1}}^{mn}\}$
   and 
   \[
   N(T, S_{k+1}(\mathbf{a}_0))=\frac{a(T^*, T)T^*}{T},
   \]
   where
       \begin{equation}\label{eq:3a:aTT}
        a(T^*, T_1)=|\{\mathbf{c}\vvert \mathbf{c}\in\mathbb{C}^{n}_p,\lcm(pT^*, T_1(\mathbf{c}))=T\}|
    \end{equation}
    is a constant depending on $T^*$ and $T$.
\end{Corollary}
\begin{proof}
    This proof is similar to Lemma~\ref{le:3A:N=a}.
\end{proof}

\begin{Corollary}\label{coro:3b:PaNTS}
There exists 
\begin{equation}\label{eq:the:Nkv2}
    \mathbf{N}_{k+l}(v+l)= \bD^l\cdot \mathbf{N}_{k}(v)
\end{equation}
for any $v>\bnu(P_1(f))$, where ${\bf N}_k(v)$ is a column vector of size $r\times 1$ with the $i$-th component $N_k(T_{k, i}(v), \mathbb{S}^{mn}_{p^k})$, 
$\{T_{k, i}(v)\}_{i=1}^r= \{T\vvert \bnu(T)=v, T\mid P_k(f)\}$ and $T_{k, i}(v)<T_{k, i+1}(v)$, 
$\bD$ is a $r\times r$ lower triangular matrix, and $l\in \mathbb{N}^+$.
\end{Corollary}
\begin{proof}
    This proof is similar to Theorem~\ref{th:3A:Nkv}.
\end{proof}

Any state $\mathbf{a}_0$ of the Cat map in $\mathbb{Z}^2_{p^k}$ corresponds to the state 
\[
\bx_0=
\begin{bmatrix}
    \mathbf{a}_0, \\
    \mathbf{C}\cdot\mathbf{a}_0
\end{bmatrix}.
\]
of the $\mathbf{A}(\mathbf{x})$. Referring to the definition of $\mathbb{S}^{4}_{p^k}$ and $\mathbb{C}_p^{mn}$ there is 
\[
\mathbb{S}^{4}_{p^k}=\left\{\bx \Bigg\vert \bx=
\begin{bmatrix}
    \mathbf{a}, \\
    \mathbf{C}\cdot\mathbf{a}
\end{bmatrix}, \mathbf{a}\in\mathbb{Z}^2_{p^k}\right\}
\]
and 
$\mathbb{C}=\mathbb{S}^{4}_{p}$. Based on Corollary~\ref{coro:3b:PaNTS},
Theorem~\ref{the:numbercycle} characterizes how the number of cycles evolves as the exponent $k$ increases by one.

\begin{theorem}
The number of cycles of period $T_{\rm c}$ of Cat map~\eqref{eq:ArnoldInteger} over $(\mathbb{Z}_{p^e}, +, \cdot)$, denoted by $N_{T_{\rm c}, e}$, 
expands in both count and period as
    \begin{equation}\label{eq:3B:Ncxe>e*} 
    p\cdot N_{T_{\rm c}, k}=N_{p\cdot T_{\rm c}, k+1}
    \end{equation}
    if $\bnu(T_c)> p$.
    \label{the:numbercycle}
\end{theorem}
\begin{proof}
Since $f(t)=t^2-(ab+2)t+1=(t-\alpha_k)(t-\alpha_k^{-1})$ in $\mathbb{R}_{k}$, it follows from the definition of $\{T_{k, i}(v)\}$ in Corollary~\ref{coro:3b:PaNTS} that there are three cases of $\{T_{k, i}(v)\}$ which are divided by the property of the root $\alpha_1$.
\begin{itemize}
    \item $\alpha_1\not\in\mathbb{Z}_{p}$: the minimal polynomic of all sequence with the initial state $\mathbf{a}\neq \bo$ is $f(t)$ and corresping period is $P_1(f)$. From the definition of $\{T_{k, i}(v)\}$, one has $\{T_{k, i}(v)\}=p^vP_1(f)$.
\item $\alpha_1\in\mathbb{Z}_{p}$ and $\alpha_1\neq \alpha^{-1}_1$: the minimal polynomic of all sequence with the initial state $\mathbf{a}\neq \bo$ maybe $f(t)$, $f_1(t)=t-\alpha_1$ and $f_2(t)=t-\alpha^{-1}_1$. Combining with $\ord_1(\alpha_1)=\ord_1(\alpha_1^{-1})$ and Lemma~\ref{le:2:P_1}, oen has $P_1(f)=P_1(f_1)=P_1(f_2)=\ord(\alpha_1)$. 
\item $\alpha_1\in\mathbb{Z}_{p}$ and $\alpha_1=\alpha^{-1}_1$: it follows from Lemma~\ref{le:2:P_1} that $P_1(f)=p\cdot \ord(\alpha_1)$. This implys that $\{T_{k, i}(v)\}=p^v\ord(\alpha_1)$. 
\end{itemize}

According to above cases, $\{T_{k, i}(v)\}$ only has one element. Hence $\mathbf{N}_k(v)$ is $1$-dimension in \eqref{eq:the:Nkv2} and $\mathbf{N}_k(v)=N(T_{k, 1}(v), \mathbb{S}_{p^k}^4)$. Since ~\eqref{eq:3a:N=bij} and ~\eqref{eq:3a:aij}, one has 
        \begin{equation*}
    \begin{split}
        a_{ij}&=\left|\left\{\mathbf{c} \Bigg \vert  \mathbf{c}\in\mathbb{C}^{4}_p,\lcm\left(T_{1,i}, \frac{T_1(\mathbf{c})}{p^{\bnu(T_1(\mathbf{c}))}}\right)=T_{1,j}\right\}\right|.\\
        &=|\mathbb{C}^{4}_p|\\
        &=p^2
    \end{split}
\end{equation*}
and $b_{ij}=p$. For any $\bnu(T_c)> p$, ~\eqref{eq:3B:Ncxe>e*} holds.
\end{proof}

\begin{theorem}\label{th:3B:TT}
When
\begin{equation}\label{eq:Cat:kf}
k\ge
\begin{cases}
     1     & \mbox{if } \alpha=\alpha^{-1};\\
 k(\alpha) & \mbox{if } \alpha\neq\alpha^{-1},    
 \end{cases}
\end{equation}
the number of cycles of period $T_{\rm c}$ of Cat map~\eqref{eq:ArnoldInteger} over $(\mathbb{Z}_{p^e}, +, \cdot)$, 
denoted by $N_{T_{\rm c}, e}$,
satisfies
\begin{equation}
    N_{T_{\rm c}, k}=N_{T_{\rm c}, k+1}, 
\end{equation}
where $\alpha$ is a root of $f(t)$ in $\mathbb{R}_1$.
\end{theorem}
\begin{proof}
Referring to Theorem~\ref{th:3A:fe*}, this theorem holds.
\end{proof}

\begin{theorem}\label{th:3B:TT}
The number of cycles of period $T_{\rm c}$ of Cat map~\eqref{eq:ArnoldInteger} over $(\mathbb{Z}_{p^e}, +, \cdot)$, 
denoted by $N_{T_{\rm c}, e}$, 
satisfies the invariance property
\begin{equation}
N_{T_{\rm c}, k}=N_{T_{\rm c}, k+1}
\end{equation}
if $k>p(p^2-1)\log_p(a\cdot b+2)$.
\end{theorem}
\begin{proof}
The proof follows the same reasoning as that of Theorem~\ref{th:3A:fe*} and is therefore omitted.
\end{proof}

\section{Conclusion}

This work presents a comprehensive study of the periodic behavior and graph structure of a class of permutation maps over the residue ring $\mathbb{Z}_{p^k}$. By leveraging the theory of linear recurrence sequences and generating functions, we explicitly characterize the least periods of such sequences and analyze their distribution across different parameter settings. The Cat map is adopted as a concrete example to verify and illustrate the theoretical results, demonstrating the practicality and relevance of the proposed framework.
The findings not only deepen our understanding of dynamical systems over modular rings but also provide a valuable foundation for future research in areas such as pseudorandom number generation, cryptographic function design, and the dynamics of modular arithmetic. Future directions may include extending the analysis to more general classes of nonlinear maps or exploring applications in secure communications and symbolic dynamics.

\bibliographystyle{IEEEtran_doi}
\bibliography{Permutation}

\renewenvironment{IEEEbiography}[1] {\IEEEbiographynophoto{#1}}  {\endIEEEbiographynophoto}

\graphicspath{{author-figures-pdf/}}

\vspace{-9mm}

\begin{IEEEbiography}{Kai Tan} received a B.Sc. degree in Mechanism Design, Manufacturing, and Automatization at the School of Mechanical Engineering, Xiangtan University, in 2015. He received his M.Sc. degree in Computer Science at the School of Computer Science, Xiangtan University in 2020.
Now, he is pursuing a PhD degree at the same institute.
His research interests include complex networks and nonlinear dynamics.
\end{IEEEbiography}

\vskip 0pt plus -1fil

\begin{IEEEbiography}{Chengqing Li}(M'07-SM'13) received his M.Sc. degree in Applied Mathematics from Zhejiang University, China, in 2005, and his
Ph.D. degree in Electronic Engineering from City University of Hong Kong in 2008.
Thereafter, he worked as a Postdoctoral Fellow at The Hong Kong Polytechnic University till September 2010.
Then, he worked at the College of Information Engineering, Xiangtan University, China. From April 2013 to July 2014, he worked at the
University of Konstanz, Germany, under the support of the Alexander von Humboldt Foundation.
He has been the dean of the School of Computer Science at Xiangtan University in China since May 2020.
He is an associate editor for the International Journal of Bifurcation and Chaos and Signal Processing.

Prof. Li focuses on the security analysis of multimedia encryption and privacy protection schemes.
Over the past twenty years, he has published more than eight papers on the subject, received over 6,200 citations, and has an h-index of 41.
He is a Fellow of the Institution of Engineering and Technology (IET).
\end{IEEEbiography}

\end{document}